\documentclass{article}

\title{Zeroth-order Optimization with Weak Dimension Dependency}
\author{%
  Pengyun Yue $\qquad$ Long Yang $\qquad$ Cong Fang $\qquad$ Zhouchen Lin \\
  School of Intelligence Science and Technology, Peking University\\
  \texttt{\{yuepy,yanglong001,fangcong,zlin\}@pku.edu.cn} \\
}
\date{}

\usepackage{bbding}

\usepackage{natbib}
\usepackage{graphicx}
\usepackage{subfigure}

\usepackage[utf8]{inputenc} 
\usepackage[T1]{fontenc}    
\usepackage{hyperref}       
\usepackage{url}            
\usepackage{booktabs}       
\usepackage{amsfonts}       
\usepackage{nicefrac}       
\usepackage{microtype}      
\usepackage{xcolor}         

\usepackage[ruled]{algorithm2e}
\usepackage{amsmath}
\usepackage{amsthm}
\usepackage{graphicx}
\usepackage{float}

\usepackage{amsmath,amsfonts,bm}
\usepackage{xfrac}

\newtheorem{theorem}{Theorem}
\newtheorem{lemma}{Lemma}
\newtheorem{definition}{Definition}
\newtheorem{proposition}{Proposition}
\newtheorem{assumption}{Assumption}
\newtheorem{corollary}{Corollary}

\newcommand{\bbeta}{\bm{\beta}}

\newcommand{\x}{{\mathbf{x}}}
\newcommand{\y}{{\mathbf{y}}}
\newcommand{\z}{{\mathbf{z}}}
\newcommand{\w}{{\mathbf{w}}}
\newcommand{\e}{{\mathbf{e}}}
\newcommand{\tx}{\tilde{\mathbf{x}}}
\newcommand{\bv}{{\mathbf{v}}}

\newcommand{\beps}{{\boldsymbol{\epsilon}}^1}
\newcommand{\bepss}{{\boldsymbol{\epsilon}}^2}

\newcommand{\mulambda}[1]{\mu_{\lambda, #1}}
\newcommand{\alambda}{a_\lambda}

\newcommand{\gauss}{{\boldsymbol{\xi}}}
\newcommand{\gz}{{\boldsymbol{\zeta}}}
\newcommand{\argmin}{\mathop{\mathrm{argmin}}}

\newcommand{\I}{{\mathbf{I}}}
\newcommand{\A}{{\mathbf{A}}}
\newcommand{\B}{{\mathbf{B}}}
\newcommand{\U}{{\mathbf{U}}}
\newcommand{\D}{{\mathbf{D}}}
\newcommand{\C}{{\mathbf{C}}}
\newcommand{\M}{{\mathbf{M}}}
\newcommand{\W}{{\mathbf{W}}}

\newcommand{\N}{{\mathbb{N}}}
\newcommand{\R}{{\mathbb{R}}}
\newcommand{\E}{{\mathbb{E}}}

\newcommand{\tr}{\mathrm{tr}}

\newcommand{\AGD}{\mathrm{AGD}}
\newcommand{\hntf}{\hat\nabla_\rho\tilde f_\delta}

\newcommand{\effdim}{\mathrm{ED}}
\newcommand{\efftrace}{\mathrm{ET}}

\newcommand{\errorA}{\epsilon_A}
\newcommand{\errorB}{\epsilon_B}
\newcommand{\errorC}{\epsilon_C}
\newcommand{\errorD}{\epsilon_D}

\newcommand{\BSa}{\mathop{\mathsf{ZHPEBinarySearch}}}
\newcommand{\BSb}{\mathop{\mathsf{ZCubicBinarySearch}}}

\newcommand{\AO}{\mathop{\mathsf{ASOE}}}
\newcommand{\ZO}{\mathop{\mathsf{ZerothOracle}}}
\newcommand{\AG}{\mathop{\mathsf{ApproximateGradient}}}
\newcommand{\ZHB}{\mathop{\mathsf{ZHB}}}
\newcommand{\RG}{\mathop{\mathsf{RandomGradient}}}

\newcommand{\rtemp}{r_{\mathrm{temp}}}
\newcommand{\ltemp}{\lambda_{\mathrm{temp}}}

\newcommand{\cO}{\mathcal{O}}
\newcommand{\tO}{\tilde{\mathcal{O}}}
\newcommand{\X}{\mathbf{X}}

\def\congfang#1{#1}

\def\original#1{}

\begin{document}

\maketitle

\begin{abstract}
Zeroth-order optimization is a fundamental research topic that has been a focus of various learning tasks, such as black-box adversarial attacks, bandits, and reinforcement learning.  However, in theory, most complexity results  assert a linear 
  dependency on the dimension of  optimization variable, which implies paralyzations of zeroth-order algorithms for high-dimensional problems and cannot explain their effectiveness  in practice.  In this paper, we present a novel zeroth-order optimization theory characterized by complexities that exhibit weak dependencies on dimensionality. The key contribution lies in the introduction of a new factor, denoted as $\effdim_{\alpha} = \sup_{\x\in \mathbb{R}^d} \sum_{i=1}^d \sigma_i^\alpha(\nabla^2 f(\x))$ ($\alpha>0$, $\sigma_i(\cdot)$ is the $i$-th singular value in non-increasing order), which effectively functions as a measure of dimensionality. The algorithms we propose demonstrate significantly reduced complexities when measured in terms of the factor  $\effdim_{\alpha}$. Specifically, we first study  a well-known zeroth-order algorithm from \cite{nesterov_random_2017} on quadratic objectives and show a complexity of $ \mathcal{O}\left(\frac{\effdim_1 }{\sigma_d}\log(1/\epsilon)\right) $ for the strongly convex setting.  For linear regression, such a complexity is dimension-free and outperforms the traditional result by a factor of $d$ under common conditions. Furthermore, we introduce novel algorithms that leverages the Heavy-ball mechanism to enhance the optimization process. By incorporating this acceleration scheme, our proposed algorithm exhibits a complexity of \congfang{$ \mathcal{O}\left(\frac{\effdim_{1/2} }{\sqrt{\sigma_d}}\cdot\log{\frac{L}{\mu}}\cdot\log(1/\epsilon)\right) $}. For linear regression, under some mild conditions, it is faster than state-of-the-art algorithms by $\sqrt{d}$. We further expand the scope of the method to encompass generic smooth optimization problems, while incorporating an additional Hessian-smooth condition. By considering this extended framework, our approach becomes applicable to a broader range of optimization scenarios. 
The resultant algorithms demonstrate remarkable complexities, with dimension-independent dominant terms that surpass existing algorithms by an order in $d$ under appropriate conditions. Our analysis lays the foundation for investigating zeroth-order optimization methods for smooth functions within high-dimensional settings.

  
\end{abstract}




\section{Introduction}
Consider the unconstrained optimization program:
\begin{equation}\label{eq:problem}
    \min_{\x\in \R^d} f(\x).
\end{equation}
We study solving \eqref{eq:problem} using \emph{zeroth-order} oracles which
only return the function value $f(\hat{\x})$ given point $\hat{\x}$ and improve such oracle complexities in searching suitable approximated solutions.

Zeroth-order optimization is a fundamental research topic serving as a prototype module for numerous tasks, including black-box optimization \citep{grill2015black}, adversarial attacks \citep{ye_hessian-aware_2019},  bandits \citep{bubeck2017kernel}, as well as reinforcement learning (RL) \citep{salimans2017evolution}. From the theoretical aspect,  one notable common feature among wide studies (see works in Section \ref{sec:related work zero}) is that the complexities of zeroth-order algorithms  have a linear dimension dependency.  For instance, consider a standard program where the objective is assumed to be $\mu$-strongly convex and have $L$-Lipschitz continuous gradients. The well-known algorithm proposed by \citet{nesterov_random_2017} called $\mathcal{RG}_\rho$   achieves a complexity of $\cO\left(\frac{dL}{\mu}\log\left(\frac{1}{\epsilon}\right)\right) $ to find an $\epsilon$-approximated solution $\tilde{\x}$ such that $f(\tilde{\x}) - \min f \leq \epsilon$. The main idea of  $\mathcal{RG}_\rho$ is solving a smoothed surrogate of $f$, whose stochastic gradients  can be efficiently computed via zeroth-order oracles of $f$, using stochastic gradient descent. Compared with the Gradient Descent algorithm,  $\mathcal{RG}_\rho$ is $d$-times slower. This result is reasonable and seems unimprovable in the worst case \congfang{because a gradient oracle offers information that can be quantified as  
a $d$-dimensional vector in contrast to $1$ of such from a function value oracle}. 

\begin{figure*}[t]
	\centering
        \vspace{-0.35cm}
	\subfigure[]{
        \begin{minipage}[t]{0.45\linewidth}
			\centering
		\includegraphics[width=2.5in]{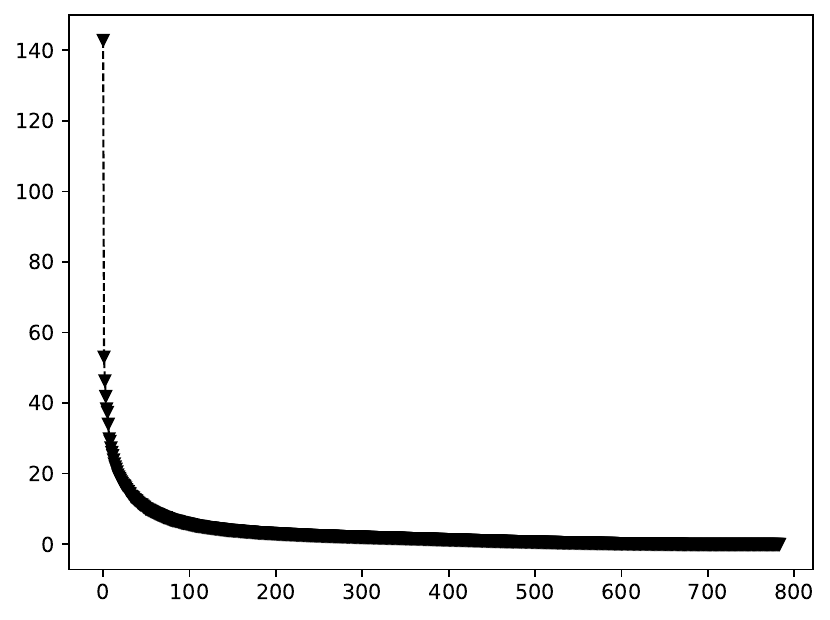}\\			\vspace{0.02cm}
		\end{minipage}%
	}~~~~~~~~~~~~
	\subfigure[]{
		\begin{minipage}[t]{0.45\linewidth}
			\centering
		\includegraphics[width=2.5in]{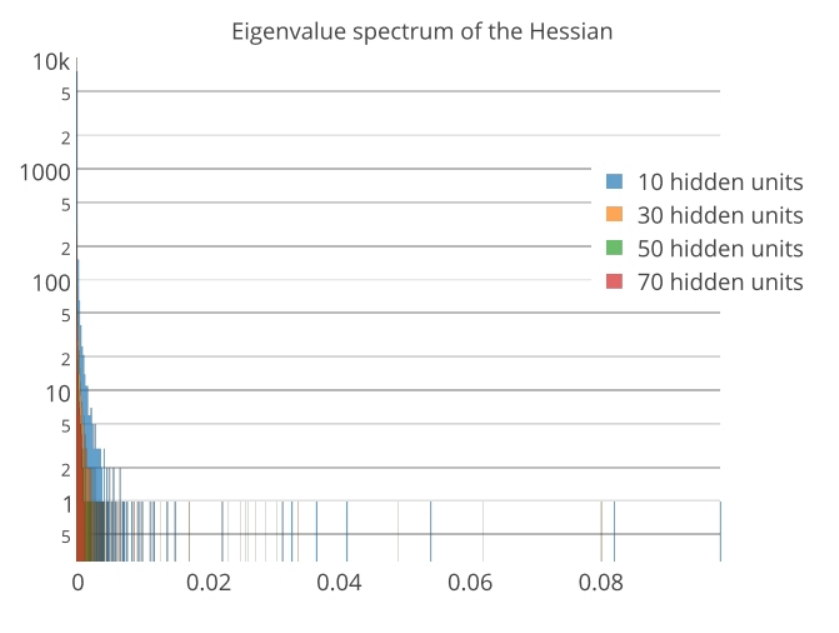}\\
			\vspace{0.02cm}
		\end{minipage}%
	}%
	\caption{(a) The eigenvalues of the Gram matrix of data on MNIST (see \citet{deng2012mnist}), which are also the eigenvalues of Hessian on the least square model  (b) The eigenvalues of a three-layer neural network on MNIST. (b) is taken directly from \citet{sagun2016eigenvalues}. }
	\vspace{-0.2cm}
	\label{fig4}
\end{figure*}

In practice, the dimension $d$ can be very large in modern real-world applications. For instance,  a  high-resolution adversarial image has thousands of pixels.  Worse still, the state numbers in contextual bandits or RL always encounter combinatorial explosions.    \original{The existing theoretical results suggest  unfortunate paralyzations of zeroth-order algorithms for high-dimensional problems. This does not square with the practice as the remarkable success of the zeroth-order algorithms has been witnessed in recent years.} The existing theoretical results indicate potential limitations of zeroth-order algorithms in high-dimensional problems, which seemingly contradict the observed success of these algorithms in practical applications over the past years. For example,  hundreds steps of $\mathcal{RG}_\rho$ suffices to find an  adversarial image \citep{ye_hessian-aware_2019}.
By estimating the objective function using (deep) neural networks, a series of RL algorithms have achieved surprising performances for decision-making \citep{mania2018simple,choromanski2018structured,salimans2017evolution}. 
These phenomena appear mysterious from a theoretical optimization view and require new analysis to understand the 
underlying reasons.

To bridge the gap between theory and practice, this paper develops a zeroth-order optimization theory which exhibits complexities with weak dimension dependencies. The underlying intuition behind our theory revolves around the introduction of an effective dimension for zeroth-order optimization. This idea has
 been widely considered in the era of machine learning and statistics (see related works in Section \ref{subsubsection:ed}), where one usually studies the required number of data  for a learning task.
  We follow a similar philosophy and show that much fewer zeroth-order oracles and iteration complexities are inherently needed when certain effective dimension is small, because only a small amount of components contribute to most of the difficulties in zeroth-order optimization. To formalize our intuition, we introduce the concept of the effective dimension in zeroth-order optimization by   $\effdim_{\alpha} = \sup_{\x\in \mathbb{R}^d} \sum_{i=1}^d \sigma_i^\alpha(\nabla^2 f(\x))$ ($\alpha>0$), where $\sigma_i(\cdot)$ is the $i$-th singular value in non-increasing order. If the objective has $L$-Lipschitz continuous gradients, one can assert  $\effdim_{\alpha} \leq d L^\alpha$. We show that under various suitable conditions, the dependence of complexities on the factor $dL^{\alpha}$ can be enhanced to $\effdim_{\alpha}$. We shall note that 
in practice one often has $\effdim_{\alpha} \ll d L^\alpha$ because the singular values of Hessian matrices often decrease very fast. See Fig.~\ref{fig4} as two examples which plot eigenvalues of Hessian matrices for a convex and a non-convex function, respectively.   To obtain a quantitative comparison between $\effdim_{\alpha}$ and $dL^\alpha$, we
 also study some realizable cases on linear regression in Section \ref{sec:effective dimension}.

Now we briefly introduce our complexity results under different settings. In Section \ref{sec:quadratic}, we first consider a basic setting where the objective is a convex quadratic function.   We study the standard $\mathcal{RG}_\rho$ algorithm and  show  complexities of $ \tO\left(\frac{\effdim_1 }{\sigma_d}\right) $ and $ \tO\left(\frac{\effdim_1 }{\epsilon}\right)$ for strongly convex and weakly convex settings, respectively, where $\tO$ hides  polylogarithmic terms. 
For linear regression where the $\ell_2$ norm of the data is normalized to a constant level,  these complexities are dimension-free and outperform the traditional results by the factor $d$. 
In Section \ref{sec:acc}, we   consider acceleration. We  propose a new Heavy-ball based algorithm, called HB-ZGD that achieves a complexity of $ \tO\left(\frac{\effdim_{1/2} }{\sqrt{\sigma_d}}\right) $ for strongly convex functions. For linear regression where the data is normalized, our algorithm is faster than the state-of-the-art algorithm \citep{nesterov_random_2017} by $\sqrt{d}$. \original{Here, the novel aspect in the complexity analysis is by using  a special Mahalanobis norm $\|\cdot\|_{\left[\nabla^2 f\right]^2}$, which { enables to show that the function value  descend faster in expectation.}} The novelty in our complexity analysis stems from the utilization of a special Mahalanobis norm, denoted as $\|\cdot\|_{\left[\nabla^2 f\right]^2}$. This unique approach allows us to demonstrate that the function value descends more rapidly, on average.
In Section \ref{sec:generic}, we extend the HB-ZGD method to encompass generic convex and non-convex optimization problems
under an additional $H$-Hessian-smooth condition. The idea is to combine HB-ZGD with cubic-regularization tricks \citep{nesterov_cubic_2006,Monteiro2013An}. \congfang{For generic convex optimization, we
obtains a complexity of $\tO\left(\effdim_{1/2}\epsilon^{-1/2} + d\epsilon^{-2/7} \right)$ against the best-known complexity of $\cO\left(d\epsilon^{-1/2}\right)$ from \cite{nesterov_random_2017}. 
For general non-convex optimization, we consider finding a second-order stationary point and establish a complexity of $\tO\left(\efftrace_{1/2}\epsilon^{-7/4} + d \epsilon^{-3/2}\right)$ against the best-known complexity of $\tO\left(d\epsilon^{-7/4}\right)$ from \cite{jin_accelerated_2017}.}

The significance of our work is two-folded. 1) By introducing an effective dimension for zeroth-order optimization, we provide a more realistic analysis for zeroth-order algorithms. Our upper bound complexities suggests that the zeroth-order optimization are usually not very hard, providing explanations for their practical successes. 2) Based on our framework, one is able to design more efficient zeroth-order algorithms under a variety of settings.  We summarize the main contributions of this work in the following. 
\begin{itemize}
    \item[(a)] We propose to use $\effdim_{\alpha} $ as the effective dimension to characterize the complexities in zeroth-order optimization. This optimization model is more close to practice.

    \item[(b)]  For quadratic objectives, we provide an improved analysis for $\mathcal{RG}_\rho$ \citep{nesterov_random_2017} and design an accelerated algorithm. We establish new dimension-independent complexities. 

    \item[(c)] For 
    generic convex and non-convex optimization, we propose provable faster algorithms  with  weak dimension dependency using the cubic regularization tricks.
\end{itemize}

\section{Related Works}
\subsection{Zeroth-order Optimization}\label{sec:related work zero}
In zeroth-order optimization, the algorithms only access the objective function value to find a designed solution. We review three main lines of research to design zeroth-order algorithms.

 The first research line is to estimate the gradient using zeroth-order oracles and then design algorithms using the techniques from first-order optimization, which is more related to our paper. One typical algorithm is the $\mathcal{RG}_\rho$ in \cite{nesterov_random_2017}. For objective functions that are $\mu$-strongly convex and have $L$-Lipschitz continuous gradient, $\mathcal{RG}_\rho$ achieves a complexity of $\cO\left(\frac{dL}{\mu}\log(1/\epsilon)\right)$ and can be accelerated using the momentum technique to achieve a complexity of $\cO\left(d\sqrt{\frac{L}{\mu}}\log(1/\epsilon)\right)$. In the generic non-convex case, \cite{nesterov_random_2017} establish a  complexity of $\cO\left(\frac{dL}{\epsilon^2}\right)$ to find an approximated first-order stationary point. There are many works that propose variants of $\mathcal{RG}_\rho$, such as proximal \citep{gasnikov2016gradient} and stochastic \citep{ghadimi2013stochastic} versions. All complexities obtained by existing  works  have a linear dimension dependency.  And we will improve the results using the proposed effective dimension.

Another popular line of research is to consider a function approximation to the objective function (see e.g. \cite{moulines2011non}). Specifically,  the way is to estimate the objective with a white-box model and balance the exploration and exploitation. The method is closely related to Bayesian optimization (see e.g. \cite{srinivas2009gaussian}).  The complexities of these algorithms are established often in a ``statistical'' style:  they depend on the Hypothesis capacity of the model and approximation error. From our view, our proposed algorithms can be recognized as using simple linear or quadratic functions to locally approximate the objective function. Essentially, we combine analytical methods in optimization and statistics.  Specifically,
 the complexities are described using some geometric characterizations that are commonly used and practical to model the objective in optimization, such as the gradient Lipchistz constant,  with a certain effective dimension, a common concept in statistical learning. In special,  we show one can often save the oracles to inaccurately estimate the gradient (locally linear approximation). It is interesting to extend our framework to study the general function approximation  and we leave such important analysis as future work. 

Last but not at least,  one more  research line is to design algorithms for more specific tasks. Typical examples are online bandits (see e.g. \cite{bubeck2017kernel}), and model-free RL (see e.g. \cite{mania2018simple}).
There are additional challenges to deal with these problems including the varying environments and randomization from policies. Many works achieve to design more efficient algorithms in terms of low regret bounds. This paper only studies vanilla zeroth-order optimization. We also leave to apply our framework on these specific learning tasks as non-trivial future works.

\subsection{Related Techniques}
\subsubsection{Effective Dimension}\label{subsubsection:ed}
  The idea of an effective dimension has long been considered in the eras of machine learning and statistics. For example, in manifold learning (see e.g. \citet{cayton2005algorithms}), one often assumes the data is embedded in a low dimensional space. In nonparametric estimation (see e.g. \citet[Chapter 13]{wainwright2019high}), one often considers the additive structure of the target function, where only a small number of dimensions combine.  More related, in linear regression, \cite{zhang_learning_2005} introduces effective dimensions on the data Gram matrices to characterize the difficulties for ridge regression and obtain complexities independent of $d$. In high-dimensional regression, one often assumes a $s$-sparse response between the output and input signals, under which much fewer observations ($s\log(d)$ in comparison $d$) are needed to determine the relations (see e.g. \citet{zhang2012general}). There are also effective dimension analysis on RL (see e.g. \cite{jin2021bellman}), whereas, our paper focuses on generic zeroth-order optimization. \cite{freund2022convergence} study Langevin sampling and show the convergence rate can be dimensional free. From our view, they in effect use  $\effdim_{2}$.  Our work generalizes theirs to the optimization field and considers much broader settings.

\subsubsection{Cubic Regularization Algorithms} Our work follows the cubic regularization tricks to work on generic optimization frameworks. Cubic regularization algorithms can be viewed as ingenious pre-conditioned Newton methods, whose updates often involve a minimization problem with the objective composed of a quadratic function with a simple third-order regularization term that can be solved using matrix inversion and a binary search. It is shown by \cite{nesterov_accelerating_2007} that the cubic regularization algorithm achieves non-asymptotic $\cO\left(\epsilon^{-1/3}\right)$ complexity for convex optimization when the objective has uniformly continuous  Hessian matrices, which outperforms the first-order algorithm with the complexity of $\cO\left(\epsilon^{-1/2}\right)$, whereas, the convergence of vanilla Newton method can  be ensured only when the initial is close to a minimizer. \cite{Monteiro2013An} further studies accelerations. The best-known complexity $\cO\left(\epsilon^{-2/7}\right)$ is obtained by tensor methods from \cite{gasnikov_optimal_2019} in the convex case, which is proved to be optimal in the worst case \citep{arjevani_oracle_2019}. In the non-convex world, \cite{nesterov_cubic_2006} show a complexity of $\cO(\epsilon^{-3/2})$ to find a second-order stationary point when treating the problem-dependent parameters as constants.

\section{Preliminary}
\subsection{Notations}
We use the convention $\mathcal{O}\left( \cdot \right)$,   $\Omega\left(\cdot\right)$, and $\Theta\left(\cdot\right)$, to denote lower, upper, both lower and upper bounds with a universal constant. $\tilde{\mathcal{O}}(\cdot)$ ignores the polylogarithmic terms.  We use $\I_d$ to denote the identity matrix in $d$-dimensional Euclidean space, and omit the subscript when $d$ is clear from the context. We use $\|\cdot\|$ to denote the operator norm of a matrix. Moreover, we use $\|\x\|$ to denote the Euclidean norm of a vector  and $\|\x\|_\A$ to denote the Mahalanobis (semi)norm  where $\A$ is a positive semi-definite matrix,  i.e.   $\|\x\|_\A = \sqrt{\x^\top\A\x}$. We use $\nabla f(\x)$ and $\nabla^2 f(\x)$ to denote the first- and second-order derivative of $f$. Moreover,  let $\x^*$ be  a minimizer of $f$ if it exists and $f^*$ be the minimum value.

\subsection{Assumptions and Definitions}
We  present some basic definitions and assumptions that are commonly used to characterize the geometry of the objective in optimization.

\begin{assumption}[Convexity]
   We say $f$ is convex if
    \begin{equation}\notag
         f(\y)\ge f(\x) + \langle \nabla f(\x), \y-\x\rangle + \frac{\mu}{2}\|\x-\y\|^2,\quad \forall \x,\y,
         \label{equ:convex}
    \end{equation}
    where $\mu\ge 0$. Moreover,  if $\mu>0$,  $f$ is said to be $\mu$-strongly convex. 
\end{assumption}

\begin{assumption}[$L$-gradient smoothness]
   We say  $f$ is  $L$-gradient smooth (or have $L$-Lipschitz continuous gradients), if 
    \begin{equation}\notag
        \|\nabla f(\x)-\nabla f(\y)\| \le L\|\x-\y\|, \quad \forall \x,\y.
    \end{equation}
\end{assumption}

\begin{assumption}[$H$-Hessian smoothness]
    We say  $f$ is  $H$-Hessian smooth (or have $H$-Lipschitz continuous Hessian matrices), if 
    \begin{equation}\notag
        \|\nabla^2 f(\x)-\nabla^2 f(\y)\|\le H\|\x-\y\|, \quad \forall \x,\y.
    \end{equation}
    \label{assumption:Hessianssmooth}
\end{assumption}

For an optimization algorithm starting at $\x^0$, we introduce the following two commonly-used quantities to describe the distance between the initial to an optimal solution.
 \begin{definition}
[$\Delta$-bounded function value] Let $\Delta {=} f(\x^0)-f^*$.
\end{definition}
\begin{definition}[$D$-bounded distance to the optimal solution]\label{def:d}Assume the minimizer of $f$ exists. Let
 $\X^*$ be the set of all minimizers. \congfang{Define $D=\inf_{\x^*\in \X^*}\sup\{\|\x-\x^*\|:f(\x)\le f(\x^0)\}$}.
\end{definition}

For convex problems, we consider finding an $\epsilon$-approximated solution  defined below:
\begin{definition}[$\epsilon$-optimal solution]
    $\x$ is an $\epsilon$-optimal solution of $f $ if $ f(\x)-f^*\le\epsilon$.
\end{definition}

For non-convex problems,  we  study finding an $(\epsilon, \cO(\sqrt{\epsilon}))$-approximated second-order stationary point with definition below:
\begin{definition}[$(\epsilon,\delta)$-SSP]\label{def:ssp}
    $\x$ is said to be an $(\epsilon,\delta)$-approximated second-order stationary point (SSP) of $f$ if it admits
    \begin{equation}\notag
    \|\nabla f(\x)\|\le\epsilon, \quad\quad \nabla^2f(\x)\succeq -\delta\I.
    \end{equation}
\end{definition}
It is known that a second-order stationary point is an optimal solution when the objective function satisfies the so-called strict-saddle condition \cite{ge2015escaping}. In our complexity analysis, we will often consider the case where $L$, $H$, $\Delta$, and $D$ are in constant level, and focus on dependencies on $\mu$, $d$,  and $\epsilon$.

\section{Effective Dimension}\label{sec:effective dimension}
Without any specification,  we always assume the objective function $f$ is  second-order derivative. We introduce the effective dimension  of zeroth-order optimization by   
$$\effdim_{\alpha} = \sup_{\x\in \mathbb{R}^d} \sum_{i=1}^d \sigma_i^\alpha(\nabla^2 f(\x))$$ 
where $\alpha>0$ and $\sigma_i(\cdot)$ is the $i$-th singular value in non-increasing order. Note that we simply  obtain $\effdim_{\alpha}$ by taking the supremum over $\mathbb{R}^d$, which is a  global quantity to  characterize the objective function.  It is possible to  consider a local  effective dimension $\effdim_{\alpha}$ and then one can choose adaptive step sizes based on the local effective dimension.  When the objective is convex, the singular values are the same as eigenvalues. For non-convex case, it is also possible to relax the singular values to   positive eigenvalues. However, we omit its analysis in this paper.

For different algorithms, we may pick different  $\alpha$. For $\mathcal{RG}_\rho$, we pick $\alpha =1$. When considering acceleration, $\alpha$ is picked as $\frac{1}{2}$.  \cite{freund2022convergence} studies Langevin algorithm and essentially pick $\alpha=2$. When the objective has $L$-Lipschitz continuous gradients,   $\effdim_{\alpha} \leq d L^\alpha$ for all $\alpha>0$. And the gap of $\effdim_{\alpha}$ to $dL^\alpha$ depends on how fast the singular values for the Hessian matrices decrease. We have a simple lemma  by supposing a descending order of singular values. 

\newcounter{pro:poly}
\setcounter{pro:poly}{\value{theorem}}
\begin{proposition}\label{pro:poly}
    Assume  for any $\x$ and $\alpha>0$, there exists constant $C>0$ and $\beta>0 $ such that $ \sigma_i(\nabla^2 f(\x)) \leq  \frac{C}{i^\beta}$ for $i\in[d]$, then we have
    \begin{equation}
        \effdim_{\alpha}\leq
\begin{cases}
\frac{2^{\alpha\beta-1}C^\alpha}{\alpha\beta-1}, \quad &\alpha\beta >1,  \quad \textit{dimensional free},\\
C^\alpha\log(2d+1), \quad &\alpha\beta =1,  \quad \textit{logarithmic growth on } d,\\
\frac{C^\alpha}{1-\alpha\beta}(d+1)^{1-\alpha \beta}, \quad &\alpha\beta <1, \quad  \textit{improve by a } \Theta\left(d^{\alpha\beta}\right) \textit{factor}. 
\end{cases}           
    \end{equation}
\end{proposition}

In the following, we show realizable cases where $\effdim_{\alpha}$ is provably smaller than $dL$, which generalizes the work from \cite{freund2022convergence}. Consider the objective admits the form  as:
\begin{align}\label{equ:ridgeseparable}
    f(\x) = \frac{1}{N}\sum_{i=1}^N q_i(\bbeta_i^\top\x),
\end{align}
with assumptions below.
\begin{assumption}\label{asp: activation}
The function $q_i\in\mathcal{C}^2$ has a bounded second derivative, i.e. $q_i''\leq L_0$ for all 
$i\in[n]$.
\end{assumption}
\begin{assumption}\label{asp: data}
For all $i\in[N]$, then norm of $\bbeta_i$ is bounded by $R$, i.e. $\|\bbeta_i\|_2\leq R$.
\end{assumption}
For linear regression,  $\bbeta_i$ is associated with the data and can achieve Assumption \ref{asp: data} by normalization, 
and $\sigma_i$ is associated with the loss function and holds Assumption \ref{asp: activation} for $\ell_2$ with $L_0=1$.  Then we have the following lemma.

\newcounter{pro:ridgeseparable}
\setcounter{pro:ridgeseparable}{\value{theorem}}
\begin{proposition}\label{pro:ridgeseparable}
    For the objective in \eqref{equ:ridgeseparable} that satisfies Assumptions \ref{asp: activation} and \ref{asp: data}, we have
    \begin{equation}
       \effdim_{\alpha} \leq
       \begin{cases}
      (L_0R)^{\alpha}, &\quad \alpha \geq 1, \quad \textit{dimensional free},\\
      (L_0R)^{\alpha}d^{1-\alpha},&\quad \alpha < 1\quad  \textit{improve by a } \Theta\left(d^{\alpha}\right) \textit{factor}.
       \end{cases}
    \end{equation}
\end{proposition}


For two-layer neural networks, we have the following proposition:
\newcounter{pro:2nn}
\setcounter{pro:2nn}{\value{theorem}}
\begin{proposition}
    Define $f(\W, \w)= \w^\top \sigma(\W^\top\x) $, where $\sigma$ is the activation function. When $\|\x\|_1 \leq r_1$, $\|\w\| \leq r_2$ and $\sigma''(x)\leq \alpha$, we have 
        $\tr\left(\nabla^2 f(\W,\w)\right) \leq \alpha r_1 r_2$.
        \label{prop:2nn}
\end{proposition}

The requirements in Proposition \ref{prop:2nn} can be met in most settings.  For deep neural networks, a similar argument can be obtained.

Finally, we note that for lots of parameterized models,  the effective dimension can be small at least when the parameter is near its optimal solution.    
This is due to the fact that under weak regular conditions, 
the fisher information $\mathcal I(\theta) = -\E \left[\frac{\partial^2}{\partial \theta^2} \log f(X;\theta)\,|\,\theta\right]=\E \left[\left(\frac{\partial}{\partial \theta} \log f(X;\theta)\right)^2\,|\,\theta\right]$. So if $\frac{\partial}{\partial \theta}\log f(X;\theta)$  is bounded, the effective dimension is also bounded. 

\section{Improved Analysis on Quadratic Minimization}
\label{sec:quadratic}
In this section, we first provide an improved analysis for zeroth-order optimization on quadratic functions.
Specifically, we assume that $f(\x)$ is a $L$-smooth and  convex quadratic function, which is  in form as
\begin{equation}\label{eq:quadractic objective}
    f(\x) = \frac12\x^\top\A\x +\mathbf{b}^\top \x.
\end{equation}
We study the quadratic function because it is already very representative since (1) in theory, it is known that most worst-case functions (lower-bound instances) in the convex optimization are exactly quadratic  (see e.g. \citet[Chapter 2]{nesterov2003introductory}); (2) in practice, quadratic functions include lots of applications in machine learning, such as least-square regression \citep{bjorck1996numerical}. Our result can be extended to work on objective functions with varying Hessian matrices. Here the Hessian matrices are needed to have a uniformly upper bound. For the sake of simplicity, we ignore such analysis.

 We focus on the standard $\mathcal{RG}_\rho$ algorithm proposed by  \cite{nesterov_random_2017}. The idea of the algorithm in \cite{nesterov_random_2017}   is to  solve a smoothed surrogate of $f$ defined as $\hat{f} = \E_{\gauss} f(\x+\rho \gauss)$, where $\rho>0$ is  picked to be small enough and $\gauss\sim N(0,\I)$.  It is shown by \cite{nesterov_random_2017} that the stochastic gradient of $\hat{f}$ can be obtained by
\begin{equation}
    \hat\nabla_\rho f(\x) = \frac{f(\x+\rho\gauss)-f(\x)}{\rho}\gauss,
    \label{def:hat}
\end{equation}
where $\gauss\sim N(0, \I)$. Therefore, one can perform stochastic gradient method to solve $\hat{f}$.

We directly relate  $ \hat\nabla_\rho f(\x)$ with $f(\x)$. In fact, by the first-order Taylor expansion on $f$,  the limit of $\hat{\nabla}_{\rho} f$ defined by $\tilde\nabla f(\x)$  with $\rho$ tending to zero admits
\begin{equation}
    \tilde\nabla f(\x)\stackrel{\triangle}{=}\lim_{\rho\to 0} \hat\nabla_\rho f(\x) = \langle \nabla f(\x), \gauss\rangle\cdot\gauss.
    \label{def:tilde}
\end{equation}
Therefore when $\rho$ is sufficiently small,  \eqref{def:hat}
returns the inner product of $\nabla f(\x)$ and a given direction $\gauss$. Moreover, by the randomness of $\gauss$,  we know that $\tilde \nabla f(\x)$ is an unbiased estimator of $\nabla f(\x)$ with  the sum of variance on each component bounded by $\Theta\left(d\|\nabla f(\x)\|^2\right)$. Specifically,

\newcounter{lem:unbiased}
\setcounter{lem:unbiased}{\value{theorem}}
\begin{lemma}
    \begin{equation}
        \E_\gauss \tilde\nabla f(\x) = \nabla f(\x)
    \end{equation}
    and 
    \begin{equation}
        \E_\gauss \|\tilde\nabla f(\x)\|^2 = \Theta(d\|\nabla f(\x)\|^2).
    \end{equation}
    \label{lem:unbiased}
\end{lemma}

A similar result of Lemma \ref{lem:unbiased} is also shown by  \cite{nesterov_random_2017}.
Lemma \ref{lem:unbiased} suggests that in order to offset the effect of variance, one needs $\Omega(d)$ estimates to obtain a unbiased estimation of $\nabla f(\x)$ with small  variance under $\ell_2$-norm. To improve the analysis on $\mathcal{RG}_\rho$, we first generalize the bound of  $\ell_2$-norm variance to the case for arbitrary Mahalanobis (semi)norm.

\newcounter{lem:descent}
\setcounter{lem:descent}{\value{theorem}}
\begin{lemma}
For symmetric matrix $\M$,
\begin{equation}
    \E_\gauss \|\tilde \nabla f(\x) \|_\M^2 \le 3\tr(\M)\|\nabla f(\x)\|^2.
\end{equation}
\label{lem:descent}
\end{lemma}
Lemma \ref{lem:descent} brings a new insight by
bridging the connections between $\mathcal{RG}_\rho$ \citep{nesterov_random_2017}
 and the effective dimension.  It suggests studying $\mathcal{RG}_\rho$ under a specific Mahalanobis (semi)norm to obtain an improved analysis. For the quadratic objective function in \eqref{eq:quadractic objective}, we can pick $\M$ as $\A$. 

Now we are ready to state the improved analysis for $\mathcal{RG}_\rho$ \citep{nesterov_random_2017}. $\mathcal{RG}_\rho$ is also shown in Algorithm \ref{alg:RG}. {{For the convenience of  later analysis in Section \ref{sec:generic}, we consider a more general zeroth-order oracle. That is, we  consider 
a zeroth-order oracle with $\delta$-adversarial noise. Specifically,  when given the input point $\tx$, such oracle returns a noisy function value  $\tilde f(\tx)$ that admits 
\begin{equation}
    \left|\tilde f(\tx)-f(\tx)\right| \le \delta,\quad \forall \x\in\R^d.\label{equ:approximateoracle}
\end{equation}
Here the noise can be adversarial.
}
We call such oracle as $\delta$-approximated zeroth-order oracle. When $\delta=0$, $\delta$-approximated zeroth-order oracle reduces to the standard zeroth-order oracle.
}
We first consider the strongly convex setting.
$\mathcal{RG}_\rho$ is shown in Algorithm \ref{alg:RG}, where we allow $\delta$-approximated zeroth-order oracle to access function value. The convergence result is shown in Theorem \ref{thm:RGconvex}.

\begin{algorithm}[t]
\caption{$\mathcal{RG}_\rho $ \citep{nesterov_random_2017}}\label{alg:RG}
\KwIn{$\x_0$}
\While{\text{stopping criterion is not met}}{
    generate  $\hat\nabla_\rho f(\x_k)$ by two queries to the function value and a Gaussian random vector in \eqref{def:hat};\\
    $\x_{k+1} \gets \x_k - h_k\hat\nabla_\rho f(\x_k)$;\\
    $k\gets k+1$;
}
\end{algorithm}

\newcounter{thm:RGconvex}
\setcounter{thm:RGconvex}{\value{theorem}}
\begin{theorem}
Suppose $f$ is a $\mu$-strongly convex quadratic function and has $L$-Lipschitz continuous gradient. The Hessian matrix of $f$ is $\A$. Let $h_k = \frac{1}{12\tr(\A)}$. Using an $\delta$-approximated zeroth-order oracle, $\{\x_k\}_{k\in \N}$ generated by $\mathcal{RG}_\rho$ satisfies
    \begin{equation}
        \begin{aligned}
            &\quad\E f(\x_{k+1}) - f^* -\frac{24\tr(\A)}{\mu}\left(C_1\rho^2+C_2\frac{\delta^2}{\rho^2}\right) \\&\le \left(1-\frac{\mu}{24\tr(\A)} \right)\left(\E f(\x_k)-f^* -\frac{24\tr(\A)}{\mu}\left(C_1\rho^2+C_2\frac{\delta^2}{\rho^2}\right) \right),
        \end{aligned}
        \label{equ:thmRGconvex5}
    \end{equation}
    where 
    \begin{equation}
    C_1 = \frac{5}{16}\tr(\A)d + \frac{5}{384}\tr(\A), \quad C_2 = \frac{d}{3\tr(\A)} + \frac{1}{72\tr(\A)},
\end{equation} 
and the expectation is taken for all the randomness in the algorithm. 
\label{thm:RGconvex}
\end{theorem}
 Theorem \ref{thm:RGconvex}  shows that  $\mathcal{RG}_\rho$ converges linearly in expectation when the hyper-parameter $\rho$ and $\delta$ are picked small enough. Moreover, in order to find an $\epsilon$-suboptimal point of $f$,  $\mathcal{RG}_\rho$  needs $\mathcal{O}\left(\frac{\effdim_1}{\mu}\log\frac{1}{\epsilon}\right)$ zeroth-order oracles and  iteration complexities. Here, we compare the result with the original one in \cite{nesterov_random_2017}, who establish a complexity of $\cO\left(\frac{dL}{\mu}\log\frac{1}{\epsilon}\right)$ in expectation.  Our analysis  is sharper than theirs  up to constants since  $\effdim_1 \leq dL$.  The rationale behind our analysis is mentioned before: in most real cases (see Fig.~\ref{fig4}), the singular values of Hessian matrices  decrease very fast. So we often have $\effdim_1\ll dL$.  To obtain a quantitative comparison between $r_1$ and $dL$, we consider linear models in \eqref{equ:ridgeseparable}   and have the following corollary.
\begin{corollary}\label{cor:stronglyconvex}
 For the objective in \eqref{equ:ridgeseparable} that satisfies Assumptions \ref{asp: activation} and \ref{asp: data} and is $\mu$-strongly convex,  $\mathcal{RG}_\rho$ finds an $\epsilon$-suboptimal solution in $\tO\left(\frac{L_0R}{\mu}\right)$ in expectation.
\end{corollary}
From Corollary \ref{cor:stronglyconvex}, treating $R$ and $L_0$ as constants,  we  establish a complexity of  $\tilde{\cO}(\mu^{-1})$ in comparison to  $\tilde{\cO}(d\mu^{-1})$ in \cite{nesterov_random_2017}. Note here  $L$ can be $\Theta(1)$.  Therefore we improve the complexity by the factor of $d$.

Now we consider the weakly convex setting. The result is shown in Theorem \ref{thm:RGweaklyconvex}.

\newcounter{thm:RGweaklyconvex}
\setcounter{thm:RGweaklyconvex}{\value{theorem}}
\begin{theorem}
    Suppose $f$ is a $L$-smooth  quadratic function whose Hessian matrix is $\A$. Using a $\delta$-approximated zeroth-order oracle, $\{\x_k\}_{k\in \N}$ generated by $\mathcal{RG}_\rho$ satisfies
    \begin{equation}
    \begin{split}
        &\qquad(k+1)(\E f(\x_{k+1})-f^*)\\ &\le k\E(f(\x_k)-f^*) + (k+1) \left(C_1\rho^2+C_2\frac{\delta^2}{\rho^2} \right)+\frac{12\tr(\A)}{k+1}\E\|\x_k-\x^*\|^2,
        \end{split}
    \end{equation}
    where the expectation is taken for all the randomness in the algorithm.
    \label{thm:RGweaklyconvex}
\end{theorem}
Theorem \ref{thm:RGweaklyconvex} establishes a complexity of $\cO\left(\frac{\effdim_1}{\epsilon}\right)$ in expectation to find an $\epsilon$-suboptimal point of $f$. Again, we compare our analysis with that in \cite{nesterov_random_2017} which achieve a complexity of $\cO\left(\frac{dL}{\epsilon}\right)$ in the same setting. Again our result is sharper than theirs up to constants. 
For linear models, when $R$ and $L_0$ are treated as constants,  the following corollary indicates that   our analysis improves the complexity by a  $d$ factor.
\begin{corollary}
   For the objective in \eqref{equ:ridgeseparable} that satisfies Assumptions \ref{asp: activation} and \ref{asp: data},  $\mathcal{RG}_\rho$ finds an $\epsilon$-suboptimal solution in $\cO\left(\frac{L_0R}{\epsilon}\right)$ oracle calls in expectation.
    \label{cor:weaklyconvex}
\end{corollary}

\begin{algorithm}[t]
    \caption{$\ZHB$: Zeroth-order Heavy-Ball algorithm.}\label{alg:FG2}
    \KwIn{$\x_0$, $L$-smooth and $\mu$-strongly quadratic function $f$}
    {$\beta \gets \sqrt{h\mu},h\gets\frac{1}{14400^2\effdim_{1/2}(f)^2}$};\\
    \While{\text{stopping criterion is not met}}{
        {$\y_n \gets \x_n+(1-\beta)\bv_n$};\\
         generate  $\hat\nabla_\rho f(\y_n)$ by two queries to the function value and a Gaussian random vector in \eqref{def:hat};\\
        {$\x_{n+1} \gets \y_n-h\hat\nabla_\rho f(\y_n)$};\\ 
        {$\bv_{n+1} \gets \x_{n+1}-\x_n$};
    }
\end{algorithm}
\section{Acceleration on Quadratic Minimization}\label{sec:acc}
It is well-known that first-order algorithms can be accelerated using the so-called momentum in convex optimization.  For example,  the earlier work from \citet{polyak1964some} shows that the Heavy-ball algorithm can achieve a faster convergence asymptotically for strongly convex functions when   Hessian matrices exist. \cite{nesterov2003introductory} proposes several acceleration schemes and first obtains the accelerated rate in the weakly convex setting by introducing the famous estimate sequence method. For quadratic objective functions,  techniques, such as Chebyshev's acceleration  and conjugate gradient (see e.g. \cite{young2014iterative}) are also applicable to reach the faster rate. 

Because $\mathcal{RG}_\rho$ can be regarded as a stochastic gradient algorithm where the variance of the stochasticity can be controlled, it is possible to perform acceleration using the technique in first-order optimization. Indeed, 
\cite{nesterov_random_2017} propose an acceleration algorithm using Nesterov's scheme in \cite{Nesterov1983AMF}. We find such a framework is not directly applicable to obtain a dimensional-independent complexity.  We instead consider a Heavy-ball based acceleration with the algorithm shown in Algorithm \ref{alg:FG2}. One can find that Algorithm \ref{alg:FG2} simply replaces the exact gradient in the Heavy-ball algorithm by a random approximation using \eqref{def:hat} and adaptively choose a different step size.  We still study the quadratic objective in \eqref{eq:quadractic objective} and first focus on strongly convex case. 
Theorem \ref{theorem:acc-zo} below summarizes our convergence result. 
\newcounter{theorem:acc-zo}
\setcounter{theorem:acc-zo}{\value{theorem}}
\begin{theorem}\label{theorem:acc-zo}
    Suppose $f$ is a $L$-smooth $\mu$-strongly convex quadratic function whose Hessian matrix is $\A$. Using an $\delta$-approximated zeroth-order oracle, if $\delta$  and $\rho$ is small enough such that
    \begin{equation}
  \congfang{      6n \left(\frac{16\delta}{\rho} +12\rho\tr(\A) \right) < 80 \left(1-\frac{\mu^{1/2}}{57600\effdim_{1/2}}\right)^{n-1}\cdot\mu\cdot(f(\x_0)-f^*),}
    \end{equation}
    $\{\x_n\}_{n\in \N}$ generated by $\ZHB$ satisfies 
    \begin{equation}
        \E f(\y_n)-f^* \le 400\left(1-\frac{\mu^{1/2}}{57600\effdim_{\frac12}} \right)^{n}\cdot \frac{L}{\mu}\cdot (f(\x_0)-f^*),
    \end{equation}
    where the expectation is taken for all the randomness in the algorithm.
\end{theorem}
The proof idea of Theorem \ref{theorem:acc-zo} is to treat each update as a vector multiplying a fixed matrix with error terms caused by the variance of estimation for $\nabla f(\x)$ and use the eigenvalues of the fixed matrix to give a convergence rate.
A similar idea also appears from \cite{jin_accelerated_2017} in generic non-convex optimization, whereas, our novel perspective is to use a special Mahalanobis norm $\|\cdot\|_{\left[\nabla^2 f\right]^2}$. 

Theorem \ref{theorem:acc-zo} shows Algorithm \ref{theorem:acc-zo} converges linearly  with a complexity of $\tO\left(\frac{\effdim_{1/2} }{\sqrt{\mu}}\right)$ in expectation
which improves the complexity of $\tO\left(d\sqrt{\frac{L}{\mu}}\right)$ in \cite{nesterov_random_2017}.  For linear models,     our analysis achieves a complexity of \congfang{$\tO\left(\frac{L_0^{1/2}R^{1/2}d^{1/2} }{\sqrt{\mu}}\right)$}, improving  the result of \cite{nesterov_random_2017} by at least  $\sqrt{d}$   when $R$ and $L_0$ are treated as constants and ignores polylogarithmic factors. Moreover,  a fully dimension-free complexity can be obtained when the eigenvalues of $\A$ decrease very fast. In particular, this requires $\sum_{k=1}^d\lambda_k^{1/2}(\A)\leq C$, which occurs, for example, when the eigenvalues decrease in $\frac{1}{k^\alpha}$ with $\alpha>2$ from Proposition \ref{pro:poly}.

We  note that using our technique, an improved analysis of the accelerated algorithm in \cite{nesterov_random_2017} can only obtain a complexity of $\tO\left(\sqrt{\frac{d\effdim_1}{\mu}} \right)$, which still has a  dimension dependency and is more costly than  $\tO\left(\frac{\effdim_{1/2}}{\sqrt{\mu}} \right)$ by a factor of $\frac{\sqrt{d\effdim_1}}{\effdim_{1/2}}$. When the eigenvalues of $\A$ decrease in $\frac{1}{k^{\alpha}}$ with $\alpha> 2$ , Algorithm \ref{alg:FG2} is provably faster by $\sqrt{d}$. 

 To extend the acceleration on the weakly convex case. We  use the standard reduction technique (see e.g. \cite{lin2015universal}) by optimizing a surrogate function:  
\begin{equation}\label{eq:surrogate}
  g(\x) =f(\x)+\frac{\epsilon}{2D^2}\| \x\|^2,  
\end{equation}
 where $\epsilon$ is a tolerant error and $D$  is defined in Definition \ref{def:d}. We have the following corollary.
\begin{corollary}\label{corollary:acc-zo}
    For convex function $f$, $\ZHB$ with regularization technique needs a complexity of $\tO\left(\effdim_{1/2}\cdot\epsilon^{-1/2}+d \right)$
     to find an $\epsilon$-suboptimal point in expectation.
\end{corollary}

\section{Accelerated Algorithms for Generic High-order Smooth Functions}\label{sec:generic}
In this section, we consider optimizing generic functions using zeroth-order oracles. To extend the analysis for quadratic minimization to a more general case, we restrict the objective to have $H$-continuous Hessian matrices. The main idea to design faster algorithms is to combine Algorithm \ref{alg:FG2} with the cubic regularization tricks \citep{Monteiro2013An, nesterov_cubic_2006}. We should mention that this paper concentrates on obtaining improved complexities. It is \emph{not} hard to simplify the designed algorithms using techniques such as \citet{jin_accelerated_2017} and \citet{fang2019sharp}. However, since the proofs are much more involved, we leave them as future works.

\subsection{Convex Case}
 We first  present an algorithm with an improved convergence rate for convex functions. 
The central idea is to adopt the large-step A-NPE method in \cite{Monteiro2013An} but considers an inexact solution for sub-problems and a binary search for hyper-parameters.  The description of the  detailed algorithm is shown in Appendix \ref{app:aconvex}.

It is shown by \cite{Monteiro2013An} that the iteration complexity of the original Large-step A-NPE can be upper bounded by $\cO\left(H^{2/7}D^{6/7}\epsilon^{-2/7}\right)$ for convex Hessian-smooth objective functions, where each update is associated with a complex cubic regularized optimization sub-problem. By inexactly solving  these subproblems with binary search and Algorithm \ref{alg:FG2}, we establish a complexity upper bound for zeroth-order algorithms. Specifically, 
\newcounter{thm:gen-convex}
\setcounter{thm:gen-convex}{\value{theorem}}
\begin{theorem}\label{thm:gen-convex}
    Assume the objective  function $f$ is convex and has $L$-continuous gradient and $H$-continuous Hessian matrices. Algorithm \ref{alg:A-HPEzo} needs
    \begin{equation}
        \tO\left(\frac{D\cdot\effdim_{1/2}}{\epsilon^{1/2}} + d\cdot D^{6/7}H^{2/7}\epsilon^{-2/7} \right)
    \end{equation}
    zeroth-order oracle calls to find an $\epsilon$-approximated solution with high probability.
\end{theorem}
 From Theorem \ref{thm:gen-convex},  if we treat $L$ and $H$ as constants, Algorithm \ref{alg:A-HPEzo} obtains a complexity of \\$\tO\left(\effdim_{1/2}\epsilon^{-1/2} + d\epsilon^{-2/7} \right)$, which is  lower than the best-known complexity of $\cO\left(d\epsilon^{-1/2}\right)$ in \cite{nesterov_random_2017} since  $\effdim_{1/2}\leq dL^{1/2}$ and usually $\effdim_{1/2}\ll d L^{1/2}$ in practice. 

 \subsection{Non-convex Case}
We consider optimizing a second-order smooth function in the general non-convex setting. For non-convex programming,  it is known that finding an approximated global minimizer for a smooth objective suffers the curse of dimensionality. We consider searching an $(\epsilon,\cO(\sqrt{\epsilon}))$-approximated second-order stationary point (see Definition~\ref{def:ssp}). Such a relaxed solution can be obtained in polynomial complexities and is already a tolerant solution for many machine learning problems such as for matrix decomposition problems \citep{ge2015escaping}.
 
We consider inexactly solving the cubic regularization algorithm in \cite{nesterov_cubic_2006} by zeroth-order oracles.  The whole algorithm is shown in Appendix \ref{app:anonconvex}. We then  provide a complexity analysis. Recall that the standard cubic regularization algorithm \cite{nesterov_cubic_2006}  finds a second-order approximated solution  in $\cO\left(H^{1/2}\Delta \epsilon^{-3/2} \right)$  
for a generic $H$-Hessian smooth function. By including the complexities to solve the subproblems,  we obtain an upper bound of zeroth-order complexity for Algorithm \ref{alg:cubiczo} in the Theorem \ref{thm:gen-nonconvex} below. 
\newcounter{thm:gen-nonconvex}
\setcounter{thm:gen-nonconvex}{\value{theorem}}
\begin{theorem}\label{thm:gen-nonconvex}
    Assume the objective  function $f$ is convex and has $L$-continuous gradients and $H$-continuous Hessian matrices. Algorithm \ref{alg:cubiczo} finds an $(\epsilon, \sqrt{H\epsilon})$-SSP of $f$ in 
    \begin{equation}
        \tO\left(\effdim_{1/2}H^{1/4}\Delta \epsilon^{-7/4} + d H^{1/2}\Delta\epsilon^{-3/2}\right)
    \end{equation}
    zeroth-order oracle calls with high probability.
\end{theorem}
 From Theorem \ref{thm:gen-nonconvex}, by treating $L$ and $H$ as constants, Algorithm \ref{alg:cubiczo} obtains a complexity of \\$\tO\left(\effdim_{1/2}\epsilon^{-7/4} + d \epsilon^{-3/2}\right)$, whereas, the best-known complexity of $\tO\left(d\epsilon^{-7/4}\right)$ from \cite{jin_accelerated_2017} in the same setting. Again Algorithm \ref{alg:cubiczo} is provably faster.

\section{Conclusion}
This paper proposes zeroth-order optimization theory with weak dimension dependency. We propose a new factor $\effdim_{\alpha}$ to characterize the complexities. Our analysis provides a new way to study zeroth-order optimization for high-dimensional problems.

\subsubsection*{Acknowledgements}
C. Fang and Z. Lin were supported by National Key R\&D Program of China (2022ZD0160301). Z. Lin was also supported by the NSF China (No. 62276004), the major key project of PCL, China (No. PCL2021A12) and Qualcomm. C. Fang was also supported by Wudao Foundation.  Thanks for  helpful discussions with Luo Luo and Haishan Ye.

\bibliographystyle{abbrvnat}
\bibliography{ref}
\newpage
\appendix
\section{Algorithms}\label{app:a}
To start with showing algorithms for generic high-order smooth function,  we first define $f_\x$ to be its second-order Taylor expansion (SOE) of $f$ at $\x$ as follows:
\begin{equation}
    f_\x(\y) = f(\x) + \langle\nabla f(\x), \y-\x\rangle + \frac{1}{2}\langle\nabla^2 f(\x)(\y-\x), \y-\x\rangle.
\end{equation}
We first design a simple algorithm  shown in Algorithm \ref{alg:quadraticsubproblem} which uses  $\Theta(1)$  zeroth-order oracles to compute a $\delta$-approximated $f_\x(\y)$ shown below.

\begin{algorithm}[!h]    \caption{$\AO(f,L,H,\x,\y,\delta)$: Compute  $\delta$-approximated  $f_\x(\y)$}\label{alg:quadraticsubproblem}
        \KwIn{an $L$-gradient Lipschitz continuous and $H$-Hessisan Lipschitz continuous function $f$, a zeroth-order oracle of $f$}
        {Denote $r=\|\y-\x\|$};\\
        {Query $f(\x)$, $f\left(\x+\frac{\delta}{Lr^2}(\y-\x)\right)$, $f\left(\x+\frac{\delta}{2Hr^3}(\y-\x)\right)$,$f\left(\x-\frac{\delta}{2Hr^3}(\y-\x)\right)$};\\
        {Approximate $f_\x(\y)$ by
        \begin{align}
            \tilde f_{\x,\delta}(\y) = f(\x) &+ \frac{Lr^2}{\delta}\left(f\left(\x+\frac{\delta}{Lr^2}(\y-\x)\right)-f(\x)\right) \\&+ \frac{2H^2r^6}{\delta^2}\left(f\left(\x+\frac{\delta}{2Hr^3}(\y-\x)\right)+f\left(\x-\frac{\delta^2}{2Hr^3}(\y-\x)\right)-2f(\x)\right);\notag
    \label{equ:quadraticapprox}
        \end{align}}
        \Return{$\tilde f_{\x,\delta}(\y)$}
\end{algorithm}

\newcounter{lem:quadraticsubproblem}
\setcounter{lem:quadraticsubproblem}{\value{theorem}}
\begin{lemma}
For function $f$ that has $L$-continuous gradient and $H$-continuous Hessian matrices, given any $\delta>0$,
    Algorithm \ref{alg:quadraticsubproblem} outputs a $\delta$-approximated $f_\x(\y)$  denoted by $\tilde f_{\x,\delta}(\y)$ such that $\left|\tilde f_{\x,\delta}(\y)- f_\x(\y)\right|\leq \delta$.
\end{lemma}

\subsection{Algorithms for Convex Optimization}\label{app:aconvex}
The proposed algorithm is shown in Algorithm \ref{alg:A-HPEzo}, where each iterate consists of inexact solving the sub-problem  by Algorithm \ref{alg:A-HPEzo-search}  and an approximated computation of the gradient using zero-order oracles by Algorithm \ref{alg:approx}. Here, Algorithm \ref{alg:A-HPEzo-search} solves the subproblem by a binary search with each step solving a quadratic minimization problem using our accelerated algorithm presented in Algorithm \ref{alg:FG2}.

\begin{algorithm}[!h]
    \caption{Inexact Large-step A-NPE with Zeroth-order Oracle}\label{alg:A-HPEzo}
    \KwIn{$\sigma_l<\sigma_u<\sigma<1$, $\sigma_l=\frac{\sigma_u}{2}$, $A_0=0$, $\errorA < \frac{D}{N^{3/2}}$, $\errorB < \frac{(\sigma-\sigma_u)^2}{2\lambda_{k+1}(L\lambda_{k+1} + 1 +(\sigma-\sigma_u)^2)\left(L+\frac{1}{\lambda_{k+1}} \right)} \cdot \left(f(\tx_k) - \min_\y \left\{f_{\tx_k}(\y)+\frac{1}{2\lambda_{k+1}}\|\y-\tx_k\|^2\right\}\right)$, $k=0$, $\lambda_0 = \frac{\sigma_l(1-\sigma^2)^{1/2}}{16DH}$;}
    \While{$k<N$}{
        {$(\y_{k+1}, a_{k+1}, \lambda_{k+1})\gets\BSa(\tx_k, H, \sigma_l,\sigma_u, A_k, \lambda_k, \errorB)$};\\
        {$\bv_{k+1}\gets\AG\left(f,\y_{k+1},\frac{\errorA}{a_{k+1}}\right)$};\\
        $A_{k+1} \gets A_k+a_{k+1}$;\\
        $\x_{k+1} \gets \x_k - a_{k+1}\bv_{k+1}$;\\
        $k \gets k+1$;
    }
\end{algorithm}
 \begin{algorithm}[!h]
     \caption{$\AG(f,\x,\errorA)$: Approximating $\nabla f(\x)$ with precision $\errorA$ for $f$ with $L$-Lipschitz gradient}\label{alg:approx}
         {$\rho\gets\frac{2\errorA}{dL}$};\\
         \For{$i\in [d]$}
             {$\bv_i \gets \frac{f(\x +\rho\e_i)-f(\x)}{\rho}$, where $\e_i$ is a vector whose $i$th coordinate is $1$ and other coordinates are $0$;}
         \Return{$\bv$}
 \end{algorithm}
\begin{algorithm}[!h]
    \caption{$\BSa(\tx_k, H, \sigma_l,\sigma_u, A_k, \lambda_k,\errorB)$: Binary search to find $\lambda_k$}\label{alg:A-HPEzo-search}
    {$\lambda_{k+1}\gets\lambda_k$};\\
    \While{True}{
        {$ a_{k+1} \gets \frac{\lambda_{k+1}+\sqrt{\lambda_{k+1}^2+4\lambda_{k+1}A_{k}}}{2}$};\\
        {$\tilde\x_k \gets \frac{A_k}{A_k+a_{k+1}}\y_k + \frac{a_{k+1}}{A_k+a_{k+1}}\x_k$};\\
        {Solve \eqref{equ:APEsubproblem1} with Algorithm \ref{alg:FG2} using Algorithm \ref{alg:quadraticsubproblem} as an oracle, and find an $\errorB$-approximated solution $\y_{k+1}$:}
        \begin{equation}
            \min_{\y\in \R^d} f_{\tilde \x_k}(\y) + \frac{1}{2\lambda_{k+1}}\|\y-\tilde \x_k\|^2.
            \label{equ:APEsubproblem1}
        \end{equation}
        {Require: $\frac{2\sigma_l}{H} \le \lambda_{k+1}\|\y_{k+1}-\tilde \x_k\|\le \frac{2\sigma_u}{H}$};\\
        \uIf{$\lambda_{k+1}\|\y_{k+1}-\tilde \x_k\|\le \frac{2\sigma_l}{H}$}
        {$\lambda_{k+1}\gets2\lambda_{k+1}$;}
        \uElseIf{$\lambda_{k+1}\|\y_{k+1}-\tilde \x_k\|\ge \frac{2\sigma_u}{H}$}
        {$\lambda_{k+1}\gets\frac{1}{2}\lambda_{k+1}$;}
        \Else
        {\Return{$(\y_{k+1}, a_{k+1}, \lambda_{k+1})$};}
    }
\end{algorithm}

\subsection{Algorithms for Non-convex Optimization}\label{app:anonconvex}
An illustration of the  algorithm is shown in Algorithm \ref{alg:cubic}. 
To solve the subproblem,  we use a binary search to determine $r_k \approx \|\x_{k+1}-\x_k\|$.  With a given $r_k$,  the subproblem can be transferred to a quadratic minimization problem and is solvable by Algorithm \ref{alg:FG2}. The whole algorithm is shown in Algorithm \ref{alg:cubiczo}, where the updates use Algorithm \ref{alg:ZCubic-search}.

\begin{algorithm}[!h]
    \caption{Illustration:  Inexact Cubic Regularization Algorithm }\label{alg:cubic}
    \While{\text{stopping criterion is not met}}{
       Approximately solve the following optimization problem using Binary Search and Algorithm \ref{alg:FG2}:
        {\begin{align}\x_{k+1}&\gets \argmin_\y f(\x_k) + \langle\nabla f(\x_k), \y-\x_k\rangle + \frac12\langle\nabla^2 f(\x_k) (\y-\x_k), \y-\x_k\rangle + \frac{H}{6}\|\y-\x_k\|^3.\notag
        \end{align}}
        $k\gets k+1$;
    }
\end{algorithm}

 \begin{algorithm}[!h]
    \caption{Inexact Cubic Regularization Algorithm with Zeroth-order Oracle}\label{alg:cubiczo}
    \KwIn{Desired accuracy $\epsilon$;}
    \While{$r_k\ge\sqrt{\frac{\epsilon}{H}}$}{
        {$(\x_{k+1}, r_{k+1})\gets \BSb(\x_k, H, r_k)$};\\
        {$k\gets k+1$};
    }
\end{algorithm}

\begin{algorithm}[!h]
    \caption{$\BSb(\x_k, H, r, \errorC, \errorD)$: Binary search to find $r_k$}\label{alg:ZCubic-search}
    {$r_{k+1}\gets r$};\\
    {$r_l\gets0,r_u\gets\infty$};\\
    \While{True}{
        {Solve \eqref{equ:Cubicsubproblem1} with Algorithm \ref{alg:FG2} using Algorithm \ref{alg:quadraticsubproblem} as an oracle, and find an $\errorC$-approximated solution $\y_{k+1}$:}
        \begin{equation}
            \min_{\y\in \R^d} f_{\x_k}(\y) + \frac{r_{k+1}H}{2}\|\y-\x_k\|^2.
            \label{equ:Cubicsubproblem1}
        \end{equation}
        \uIf{$\|\y_{k+1}-\x_k\| \le r_{k+1}$}{
        {$r_u\gets r_{k+1}$};\\
        {$r_{k+1}\gets \frac{r_{k+1}}{2}$};
        }
        \ElseIf{$\|\y_{k+1}-\x_k\| > r_{k+1}$}{
        {$r_l\gets r_{k+1}$};\\
        {$r_{k+1}\gets 2r_{k+1}$};
        }
        \If{($r_l>0 $ and $r_u<\infty$) or $r_u<\errorD$}
        {break;}
    }
    \While{$r_u-r_l\ge \errorD$}{
    {$r_{k+1} \gets \frac{r_u+r_l}{2}$};\\
    {Solve \eqref{equ:Cubicsubproblem1} with Algorithm \ref{alg:FG2} using Algorithm \ref{alg:quadraticsubproblem} as an oracle, and find an $\errorC$-approximated solution $\y_{k+1}$};\\
    \uIf{$\|\y_{k+1}-\x_k\| \le r_{k+1}$} {$r_u\gets r_{k+1}$;}  
    \ElseIf{$\|\y_{k+1}-\x_k\| > r_{k+1}$} {$r_l\gets r_{k+1}$;}
    }
    {$r_{k+1} \gets r_u$};\\
    {Solve \eqref{equ:Cubicsubproblem1} with Algorithm \ref{alg:FG2} using Algorithm \ref{alg:quadraticsubproblem} as an oracle, and find an $\errorC$-approximated solution $\y_{k+1}$};\\
    {\Return{$(\y_{k+1}, r_{k+1})$}};
\end{algorithm}

\newpage
\section{Proofs of Lemmas about Estimating  Gradients}
In this section, we present the proof of Lemma \ref{lem:unbiased} and Lemma \ref{lem:descent}, and other technical lemmas about the properties of $\hntf$. 

\newcounter{lem:unbiasedtemp}
\setcounter{lem:unbiasedtemp}{\value{theorem}}
\setcounter{theorem}{\value{lem:unbiased}}
\begin{lemma}
    \begin{equation}
        \E_\gauss \tilde\nabla f(\x) = \nabla f(\x)
    \end{equation}
    and 
    \begin{equation}
        \E_\gauss \|\tilde\nabla f(\x)\|^2 = \Theta(d\|\nabla f(\x)\|^2).
    \end{equation}
\end{lemma}
\setcounter{theorem}{\value{lem:unbiasedtemp}}

\begin{proof}
For the expectation, we have
\begin{equation}
    \begin{aligned}
        \E_\gauss \tilde\nabla f(\x)&= \E_\gauss \langle\tilde\nabla f(\x),\gauss\rangle\cdot\gauss = \E_\gauss \gauss\gauss^\top\nabla f(\x)\\
        &= \I \nabla f(\x) = \nabla f(\x).
    \end{aligned}
\end{equation}
For the sum of second-order moment, for an arbitrary symmetric matrix $\M$, we have
\begin{equation}
    \begin{aligned}
        \E_\gauss \|\tilde\nabla f(\x)\|_\M^2 &= \E_\gauss \|\langle\nabla f(\x), \gauss\rangle\cdot \gauss\|_\M^2\\
        &= \E_\gauss \nabla f(\x)^\top\gauss\gauss^\top\M\gauss\gauss^\top\nabla f(\x)\\
        &= \nabla f(\x)^\top\E_\gauss \left[\gauss\gauss^\top\M\gauss\gauss^\top\right] \nabla f(\x).
    \end{aligned}
    \label{equ:lemdescent1}
\end{equation}
Let $\M = \U^\top\D \U$ be the eigenvalue decomposition of $\M$ where $\D = \mathrm{diag}\{b_1,\cdots,b_d\}$ is a diagonal matrix, and $\boldsymbol{\zeta}=\U\gauss$ be a random variable. We have $\boldsymbol{\zeta}\sim N(0,\I)$ because $\U$ is a orthogonal matrix, and
\begin{equation}
    \begin{aligned}
   \E_\gauss \left[\gauss\gauss^\top\M\gauss\gauss^\top \right] \overset{a}= \E_{\boldsymbol{\zeta}} \left[\U^\top\gz\gz^\top\D\gz\gz^\top\U \right]
        = \U^\top\E_\gz\left[\sum_{i=1}^d b_i\gz_i^2 \cdot\gz\gz^\top \right]\U\\  \vspace{-0.15in}
        \overset{b}= \U^\top\left(\sum_{i=1}^d b_i\cdot \I +2\D\right)\U
        \overset{c}=  \tr(\M)\cdot \I + 2\M,
    \end{aligned}
    \label{equ:lemdescent2}
\end{equation}
where in $\overset{a}= $, we introduce $\boldsymbol \zeta = \U\gauss$, then $\boldsymbol \zeta$ also follows from standard Gaussian distribution by the rotational invariance,   in $\overset{b}= $, we use the second and fourth order moment of standard Gaussian variables: $\E \boldsymbol\zeta_i^2=1$, $\E \boldsymbol\zeta_i^4=3$, and in $\overset{c}= $, we use $\tr(\M) = \tr(\U^\top \D\U)  =\tr( \D\U\U^\top)=\tr(\D)$.

When $\M=\I$, we have
\begin{equation}
    \E_\gauss \|\tilde\nabla f(\x)\|^2 = (d+2) \|\nabla f(\x)\|^2 = \Theta(d\|\nabla f(\x)\|^2).
\end{equation}
\end{proof}

\newcounter{lem:descenttemp}
\setcounter{lem:descenttemp}{\value{theorem}}
\setcounter{theorem}{\value{lem:descent}}
\begin{lemma}
For symmetric matrix $\M$,
\begin{equation}
    \E_\gauss \|\tilde \nabla f(\x) \|_\M^2 \le 3\tr(\M)\|\nabla f(\x)\|^2.
    \label{equ:lemdescent}
\end{equation}
\end{lemma}
\setcounter{theorem}{\value{lem:descenttemp}}

\begin{proof} 
By \eqref{equ:lemdescent2}, \eqref{equ:lemdescent1}, and the fact that $\tr(\M)\cdot \I\succeq L\I\succeq \M$ we have \eqref{equ:lemdescent}.
\end{proof}



\begin{lemma} 
Let $\tilde f_\delta$ be a $\delta$-approximated estimate of $f$. If $\hat\nabla_{\rho}\tilde f_\delta(\x)$ and $\tilde\nabla f(\x)$ are generated by the same Gaussian random variable, 
\begin{equation}
    \E_{k+1} \|\hat\nabla_{\rho}\tilde f_\delta(\x) - \tilde\nabla f(\x)\|_\B^2\le \frac{8\delta^2}{\rho^2}\tr(\B) + \frac{15\rho^2}{2}\tr(\A)^2\tr(\B),
\end{equation}
where $\B$ is an arbitrary positive semi-definite symmetric matrix.
\label{lem:error}
\end{lemma}

\begin{proof}
    By the definition of $\hat\nabla$ in \eqref{def:hat}, $\tilde\nabla$ in \eqref{def:tilde}, and $\tilde f_\delta$, we have
    \begin{align}
&\quad\hat\nabla_{\rho}\tilde f_\delta(\x) - \tilde\nabla f(\x) \label{equ:lemerror1}\\
            &= \left( \frac{\tilde f_\delta (\x + \rho\gauss) - \tilde f_\delta(\x)}{\rho} - \left\langle\nabla f(\x), \gauss \right\rangle\right)\cdot\gauss\notag\\
            &= \left( \left[\frac{\tilde f_\delta (\x + \rho\gauss)-\tilde f_\delta(\x)}{\rho} - \frac{f(\x+\rho\gauss) - f(\x)}{\rho} \right] + \left[\frac{f(\x+\rho\gauss) - f(\x)}{\rho} - \langle\nabla f(\x), \gauss\rangle\right]\right)\cdot\gauss. \notag
    \end{align}
    
    By \eqref{equ:lemerror1}, we have
    \begin{equation}
        \begin{aligned}
            &\quad \E_{k+1}\left\|\hat\nabla_{\rho}\tilde f_\delta(\x) - \tilde\nabla f(\x) \right\|_\B^2\\
            &\le 2\E_{k+1}\left(\frac{\tilde f_\delta (\x + \rho\gauss)-\tilde f_\delta(\x)}{\rho} - \frac{f(\x+\rho\gauss) - f(\x)}{\rho}\right)^2\cdot \|\gauss\|_\B^2 \\
            &\quad + 2\E_{k+1} \left(\frac{f(\x+\rho\gauss) - f(\x)}{\rho} - \langle\nabla f(\x), \gauss\rangle \right)^2\cdot\|\gauss\|_\B^2\\
            &\le \frac{8\delta^2}{\rho^2}\E_{k+1}\|\gauss\|_\B^2 + \frac12\E_{k+1} \rho^2\|\gauss\|^4_\A\|\gauss\|^2_\B\\
            &\le \frac{8\delta^2}{\rho^2}\tr(\B) + \frac{15\rho^2}{2}\tr(\A)^2\tr(\B).
        \end{aligned}
    \end{equation}
\end{proof}

\newcounter{temp}
\setcounter{temp}{\value{theorem}}
\setcounter{theorem}{\value{thm:RGconvex}}
\section{Proofs of Theorems \ref{thm:RGconvex} and \ref{thm:RGweaklyconvex}}
\begin{theorem}
Suppose $f$ is a $\mu$-strongly convex quadratic function and has $L$-Lipschitz continuous gradient. The Hessian matrix of $f$ is $\A$. Let $h_k = \frac{1}{12\tr(\A)}$. Using an $\delta$-approximated zeroth-order oracle, $\{\x_k\}_{k\in \N}$ generated by $\mathcal{RG}_\rho$ satisfies
    \begin{equation}
        \begin{aligned}
            &\quad\E f(\x_{k+1}) - f^* -\frac{24\tr(\A)}{\mu}\left(C_1\rho^2+C_2\frac{\delta^2}{\rho^2}\right) \\&\le \left(1-\frac{\mu}{24\tr(\A)} \right)\left(\E f(\x_k)-f^* -\frac{24\tr(\A)}{\mu}\left(C_1\rho^2+C_2\frac{\delta^2}{\rho^2}\right) \right),
        \end{aligned}
    \end{equation}
    where 
    \begin{equation}
    C_1 = \frac{5}{16}\tr(\A)d + \frac{5}{384}\tr(\A), \quad C_2 = \frac{d}{3\tr(\A)} + \frac{1}{72\tr(\A)},
    \label{equ:Cdef}
\end{equation} 
and the expectation is taken for all the randomness in the algorithm. 
\end{theorem}
\setcounter{theorem}{\value{temp}}

\begin{proof}
 Because $f$ is quadratic, we have
    \begin{equation}
        f(\x_{k+1}) = f(\x_k) + \langle \nabla f(\x_k), \x_{k+1}-\x_k\rangle + \frac{h_k^2}{2}\|\hat\nabla_\rho \tilde f_\delta(\x_k)\|_{\A}^2.
        \label{equ:thmRG1}
    \end{equation}

    Taking expectation with respect to the randomization of $\hat\nabla _\rho \tilde f_\delta(\x_k)$ on both sides of \eqref{equ:thmRG1}, we get
    \begin{align}
            &\E_{k+1} f(\x_{k+1})\notag\\
            \le& f(\x_k) - h_k\langle\nabla f(\x_k), \E_{k+1} \hntf(\x_k)  \rangle + \frac{h_k^2}{2} \E_{k+1}\|\hntf(\x_k) \|_\A^2\notag\\
            \le& f(\x_k) - h_k\langle\nabla f(\x_k), \E_{k+1} \tilde\nabla f(\x_k) \rangle - h_k\left\langle\nabla f(\x_k), \E_{k+1} \left[\hntf(\x_k) - \tilde\nabla f(\x_k)\right] \right\rangle \notag \\
            &\quad+ h_k^2\E_{k+1}\|\tilde\nabla f(\x_k)\|_\A^2 + h_k^2 \E_{k+1}\|\hntf(\x_k) - \tilde \nabla f(\x_k)\|_\A^2.
        \label{equ:thmRG2}
    \end{align}

    Using Lemma \ref{lem:unbiased} and Lemma \ref{lem:descent} in \eqref{equ:thmRG2}, we have
    \begin{equation}
        \begin{aligned}
            &\E_{k+1} f(\x_{k+1})\notag\\ \le& f(\x_k) - h_k\|\nabla f(\x_k)\|^2 + h_k^2\cdot 3\tr(\A)\E_{k+1}\|\nabla f(\x_k)\|^2\\
            &\quad -h_k\left\langle\nabla f(\x_k), \E_{k+1} \left[\hntf(\x_k) - \tilde\nabla f(\x_k)\right] \right\rangle + h_k^2\E_{k+1}\|\hntf(\x_k)-\tilde\nabla f(\x_k)\|^2_\A\\
            \le& f(\x_k) - h_k\|\nabla f(\x_k)\|^2 + h_k^2\cdot 3\tr(\A)\E_{k+1}\|\nabla f(\x_k)\|^2\\
            &\quad +\frac{h_k}{2}\|\nabla f(\x_k)\|^2+\frac{h_k}{2} \E_{k+1} \|\hntf(\x_k) - \tilde\nabla f(\x_k)\|^2 \\
            &\quad+ h_k^2\E_{k+1}\|\hntf(\x_k)-\tilde\nabla f(\x_k)\|^2_\A.
        \end{aligned}
    \end{equation}

    By Lemma \ref{lem:error}, we have
    \begin{equation}
        \begin{aligned}
            \E_{k+1}f(\x_{k+1}) &\le f(\x_k) -\frac{h_k}{2}\|\nabla f(\x_k)\|^2 + 3h_k^2\tr(\A)\E_{k+1}\|\nabla f(\x_k)\|^2\\
            &\quad +\frac{h_k}{2}\cdot\left(\frac{8\delta^2}{\rho^2}d+\frac{15\rho^2}{2}\tr(\A)^2d \right) + h_k^2\cdot\left(\frac{8\delta^2}{\rho^2}\tr(\A)+\frac{15\rho^2}{2}\tr(\A)^3 \right).
        \end{aligned}
    \end{equation}

    By the definition of $h_k$, we have
    \begin{equation}
        \begin{aligned}
            \E_{k+1}f(\x_{k+1}) &\le f(\x_k) -\frac{1}{48\tr(\A)}\|\nabla f(\x_k)\|^2\\
            &\quad + \frac{1}{24}\left(\frac{8\sigma^2d}{\rho^2\tr(\A)} + \frac{15\rho^2\tr(\A)d}{2}\right)+\frac{1}{576}\left(\frac{8\delta^2}{\rho^2\tr(\A)} + \frac{15\rho^2\tr(\A)}{2}\right)\\
            &= f(\x_k) -\frac{1}{48\tr(\A)}\|\nabla f(\x_k)\|^2\\
            &\quad +\rho^2\left(\frac{5}{16}\tr(\A)d + \frac{5}{384}\tr(\A) \right) + \frac{\sigma^2}{\rho^2}\left(\frac{d}{3\tr(\A)} + \frac{1}{72\tr(\A)} \right).
        \end{aligned}
        \label{equ:thmRG3}
    \end{equation}
    
    By \eqref{equ:thmRG3}, we have
    \begin{equation}
        \begin{aligned}
            \E_{k+1} f(\x_{k+1}) &\le f(\x_k) -\frac{1}{48\tr(\A)}\|\nabla f(\x_k)\|^2\\
            &\quad+C_1\rho^2+C_2\frac{\delta^2}{\rho^2},
        \end{aligned}
        \label{equ:thmRGconvex1}
    \end{equation}
    where $C_1$ and $C_2$ are constants depending only on $\tr(\A)$ and $d$, defined in \eqref{equ:Cdef}. By the strong convexity of $f$, we have
    \begin{equation}
        f(\x_{k})\le f^* + \langle\nabla f(\x_k), \x_k-\x^* \rangle -\frac{\mu}{2} \|\x_k-\x^*\|^2.
    \end{equation}
    By Cauchy-Schwartz inequality, we have
    \begin{equation}
        \langle\nabla f(\x_k), \x_k-\x^* \rangle \le \frac{1}{2\mu}\|\nabla f(\x_k)\|^2 + \frac{\mu}{2}\|\x_k-\x^*\|^2.
    \end{equation}
    Therefore, 
    \begin{equation}
        f(\x_{k})\le f^* + \frac{1}{2\mu}\|\nabla f(\x_k)\|^2.
        \label{equ:thmRGconvex2}
    \end{equation}
    Plugging \eqref{equ:thmRGconvex1} into \eqref{equ:thmRGconvex2}, we have
    \begin{equation}
        \begin{aligned}
            \E_{k+1}f(\x_{k+1}) &\le f(\x_k)-\frac{\mu}{24\tr(\A)}\left(f(\x_k)-f^* \right) \\
            &\quad + C_1\rho^2+C_2\frac{\delta^2}{\rho^2}.
        \end{aligned}
        \label{equ:thmRGconvex3}
    \end{equation}
    Taking full expectation to \eqref{equ:thmRGconvex3}, we have
    \begin{equation}
        \begin{aligned}
            \E f(\x_{k+1}) - f^* &\le \left(1-\frac{\mu}{24\tr(\A)} \right)\left(\E f(\x_k)-f^*\right) + C_1\rho^2 + C_2\frac{\delta^2}{\rho^2}.
        \end{aligned}
        \label{equ:thmRGconvex4}
    \end{equation}
    By \eqref{equ:thmRGconvex4}, we have
    \begin{equation}
        \begin{aligned}
            &\qquad\E f(\x_{k+1}) - f^* -\frac{24\tr(\A)}{\mu}\left(C_1\rho^2+C_2\frac{\delta^2}{\rho^2}\right)\\ &\le \left(1-\frac{\mu}{24\tr(\A)} \right)\left(\E f(\x_k)-f^* -\frac{24\tr(\A)}{\mu}\left(C_1\rho^2+C_2\frac{\delta^2}{\rho^2}\right) \right).
        \end{aligned}
    \end{equation}
\end{proof}

\setcounter{temp}{\value{theorem}}
\setcounter{theorem}{\value{thm:RGweaklyconvex}}
\begin{theorem}
    Suppose $f$ is a $L$-smooth  quadratic function whose Hessian matrix is $\A$. Using a $\delta$-approximated zeroth-order oracle, $\{\x_k\}_{k\in \N}$ generated by $\mathcal{RG}_\rho$ satisfies
    \begin{equation}
    \begin{split}
        &\qquad(k+1)(\E f(\x_{k+1})-f^*)\\ &\le k\E(f(\x_k)-f^*) + (k+1) \left(C_1\rho^2+C_2\frac{\delta^2}{\rho^2} \right)+\frac{12\tr(\A)}{k+1}\E\|\x_k-\x^*\|^2,
        \end{split}
    \end{equation}
    where the expectation is taken for all the randomness in the algorithm.
\end{theorem}
\setcounter{theorem}{\value{temp}}

\begin{proof}
    By \eqref{equ:thmRG3}, we have
    \begin{equation}
        \begin{aligned}
            \E_{k+1} f(\x_{k+1}) &\le f(\x_k) -\frac{1}{48\tr(\A)}\|\nabla f(\x_k)\|^2\\
            &\quad+C_1\rho^2+C_2\frac{\delta^2}{\rho^2},
        \end{aligned}
        \label{equ:thmRGwconvex1}
    \end{equation}
    where $C_1$ and $C_2$ are constants depending only on $\|\A\|_*$ and $d$, defined in \eqref{equ:Cdef}. By the convexity of $f$, we have
    \begin{equation}
        f(\x_{k})\le f^* + \langle\nabla f(\x_k), \x_k-\x^* \rangle.
        \label{equ:thmRGwconvexp}
    \end{equation}
    By Cauchy-Schwartz inequality, we have
    \begin{equation}
        \langle\nabla f(\x_k), \x_k-\x^* \rangle \le \frac{1}{48\tr(\A)(k+1)}\|\nabla f(\x_k)\|^2 + 12\tr(\A)(k+1)\|\x_k-\x^*\|^2.
        \label{equ:thmRGwconvex2}
    \end{equation}
    Using \eqref{equ:thmRGwconvexp} and \eqref{equ:thmRGwconvex2}, we have
    \begin{equation}
    \begin{split}
        f(\x_k)&\le f^* + \frac{k+1}{48\tr(\A)}\|\nabla f(\x_k)\|^2 + \frac{12\tr(\A)}{k+1}\|\x_k-\x^*\|^2\\
        &\stackrel{\eqref{equ:thmRGwconvex1}}{\le} f^* + (k+1)\left(f(\x_k)-\E_{k+1} f(\x_{k+1}) + C_1\rho^2+C_2\frac{\delta^2}{\rho^2} \right)\\
        &\quad+ \frac{12\tr(\A)}{k+1}\|\x_k-\x^*\|^2.
    \end{split}
    \end{equation}
    Therefore, we have
    \begin{equation}
    \begin{split}
        &\quad (k+1)(\E_{k+1}f(\x_{k+1})-f^*)\\
        &\le k(f(\x_k)-f^*) + (k+1) \left(C_1\rho^2+C_2\frac{\delta^2}{\rho^2} \right)+\frac{12\tr(\A)}{k+1}\|\x_k-\x^*\|^2.
        \end{split}
    \end{equation}
\end{proof}

\section{Proof of Theorem \ref{theorem:acc-zo}}
\setcounter{temp}{\value{theorem}}
\setcounter{theorem}{\value{theorem:acc-zo}}
\begin{theorem}
    Suppose $f$ is a $L$-smooth $\mu$-strongly convex quadratic function whose Hessian matrix is $\A$. Using an $\delta$-approximated zeroth-order oracle, if $\delta$  and $\rho$ is small enough such that
    \begin{equation}
  \congfang{      6n \left(\frac{16\delta}{\rho} +12\rho\tr(\A) \right) < 80 \left(1-\frac{\mu^{1/2}}{57600\effdim_{1/2}}\right)^{n-1}\cdot\mu\cdot(f(\x_0)-f^*),}
    \end{equation}
    $\{\x_n\}_{n\in \N}$ generated by $\ZHB$ satisfies 
    \begin{equation}
        \E f(\y_n)-f^* \le 400\left(1-\frac{\mu^{1/2}}{57600\effdim_{\frac12}} \right)^{n}\cdot \frac{L}{\mu}\cdot (f(\x_0)-f^*),
    \end{equation}
    where the expectation is taken for all the randomness in the algorithm.
\end{theorem}
\setcounter{theorem}{\value{temp}}

\begin{proof}
    Let $\z_{k+1} = \begin{bmatrix}
        \x_{k+1}\\
        \x_k
    \end{bmatrix}$. The iterations of $\ZHB$ can be written as
    \begin{equation}
        \z_{k+1} = \begin{bmatrix}
            (2-\beta)(\I-h\A)&-(1-\beta)(\I-h\A)\\
            \I&\mathbf 0
        \end{bmatrix}\z_k+h\beps_k+h\bepss_k\\
        \stackrel{\triangle}{=}\B\z_k+h\beps_k+h\bepss_k,\\
        \label{equ:zdef}
    \end{equation}
    
    where $\beps_k = \begin{bmatrix}(\I-\gauss\gauss^\top)\A\y_k\\\mathbf 0\end{bmatrix}$, and $\bepss_k = \begin{bmatrix}\tilde\nabla f(\x_k)-\hntf(\x_k)\\\mathbf 0\end{bmatrix}$. $\beps_k$ represents the error of estimating $\nabla f(\x_k)$ with $\tilde\nabla (\x_k)$, and $\bepss_k$ represents the error of estimating $\ZHB$ with $\hntf(\x_k)$. 

    By induction on $k$, we have
    \begin{equation}
        \z_{n} = \B^n \z_0 + h\sum_{k=0}^{n-1} \B^{n-k-1}\beps_{k} + h\sum_{k=0}^{n-1} \B^{n-k-1}\bepss_{k}.
        \label{equ:recursive}
    \end{equation}

    Without loss of generality, we assume that $\x^*=\mathbf 0$. We estimate the distance to the optimal solution by the $\A^2$ norm of $\x_k$. To compute $\|\x_k\|_{\A^2}$, we decompose $\x_k$ into eigen-directions of $\A$, and $\B$ can be decomposed into $2\times2$ matrices. For an eigen-direction with eigenvalue $\lambda$, the update of $\AGD$ can be written as follows:
    
    \begin{equation}
    \begin{aligned}
        \begin{bmatrix}
            x_{k+1}\\x_k
        \end{bmatrix} &= 
        \begin{bmatrix}
            (2-\beta)(1-h\lambda)&-(1-\beta)(1-h\lambda)\\
            1&0
        \end{bmatrix}\begin{bmatrix}
            x_{k}\\x_{k-1}
        \end{bmatrix}
        +h\begin{bmatrix}
            \epsilon\\0
        \end{bmatrix}\\
        &\stackrel{\triangle}{=}\B_\lambda\begin{bmatrix}
            x_{k}\\x_{k-1}
        \end{bmatrix}
        +h\begin{bmatrix}
            \epsilon\\0
        \end{bmatrix}.
    \end{aligned}
    \end{equation}
    Let $\mu_1$ and $\mu_2$ be the eigenvalues of $\B_\lambda$. By the Lemma 19 of \cite{jin_accelerated_2017}, we can write the eigen-decomposition of $\B_\lambda$ as
    \begin{equation}
        \B_\lambda = \frac{1}{\mu_1-\mu_2}\begin{bmatrix}
            \mu_1&\mu_2\\
            1&1
        \end{bmatrix}
        \begin{bmatrix}
            \mu_1&0\\0&\mu_2
        \end{bmatrix}
        \begin{bmatrix}
            1&-\mu_2\\-1&\mu_1
        \end{bmatrix}.
    \end{equation}

    Let $\C = \begin{bmatrix}
        \A^2&\mathbf 0\\\mathbf 0&\mathbf \A^2
    \end{bmatrix}$. By \eqref{equ:recursive}, We have

    \begin{equation}
        \begin{aligned}
            \E \|\z_{n}\|_\C^2
            &\le 3\|\B^n\z_0\|_{\C}^2 +3\E \left\|\sum_{k=0}^{n-1} \B^{n-k-1} \beps_k\right\|_\C^2 + 3\E \left\|\sum_{k=0}^{n-1} \B^{n-k-1} \bepss_k\right\|_\C^2\\
        \end{aligned}
        \label{equ:znexpansion}
    \end{equation}
    
    First we tackle the $\beps_k$ terms. 
    \begin{equation}
        \begin{aligned}
            &\quad\E \left\|\sum_{k=0}^{n-1} \B^{n-k-1} \beps_k\right\|_\C^2 \\
            &= \sum_{k=0}^{n-1} \E_k \left\| \B^{n-k-1} \beps_k \right\|_\C^2\\
            &= \sum_{k=0}^{n -1}\E_i \begin{bmatrix}
                \y_k^\top\A^\top(\I-\gauss_k\gauss_k^\top)&\mathbf 0
            \end{bmatrix} \B^{(n-k-1)T} \C \B^{n-k-1}\begin{bmatrix}
                (\I-\gauss_k\gauss_k^\top)\A\y_k\\\mathbf 0
            \end{bmatrix}\\
            &\stackrel{\text{Lemma \ref{lem:descent}}}{\le} 3\sum_{k=1}^n \tr\left(\B^{(n-k-1)T} \C \B^{n-k-1}\right)\cdot \|\y_k\|_{\A^2}^2.
        \end{aligned}
    \end{equation}
    In order to estimate $\tr\left(\B^{(n-k-1)T} \C \B^{n-k-1}\right)$, we consider blocks of $\B$ with respect to eigen-directions of $\A$. The contribution of an eigen-direction with eigenvalue $\lambda$ in the trace is
    \begin{equation}
        \begin{aligned}
        &\quad\tr \left(\B_\lambda^{(n-k-1)\top}\cdot\begin{bmatrix} \lambda^2&0\\0&\lambda^2\end{bmatrix} \B^{(n-k-1)}\right)\\
        &= \lambda^2\left(\left\|\begin{bmatrix}
            1&0
        \end{bmatrix}\B_\lambda^{n-k-1}\right\|^2 + \left\|\begin{bmatrix}
            0&1
        \end{bmatrix}\B_\lambda^{n-k-1}\right\|^2\right)
        \end{aligned}
        \label{equ:trace}
    \end{equation}
    By Lemma 19 of \cite{jin_accelerated_2017}, the last line in \eqref{equ:trace} equals to
    \begin{equation}
        \begin{aligned}
        &\lambda^2 \left\|\begin{bmatrix} \sum_{i=0}^{n-k-1} \mulambda{1}^i\mulambda{2}^{n-k-1-i} & -\mulambda{1}\mulambda{2}\sum_{i=0}^{n-k-2} \mulambda{1}^i\mulambda{2} ^{n-k-2-i}  \end{bmatrix}\right\|^2 \\
        &\quad+\lambda^2\left\|\begin{bmatrix}\sum_{i=0}^{n-k-2} \mulambda{1}^i\mulambda{2}^{n-k-2-i} & -\mulambda{1}\mulambda{2}\sum_{i=0}^{n-k-3} \mulambda{1}^i\mulambda{2}^{n-k-3-i}\end{bmatrix}\right\|^2.
        \end{aligned}
    \end{equation}
    Define $\alambda=|\mulambda{1}| = \sqrt{(1-\beta)(1-h\lambda)}$. By the choice of $\beta$, we have $a_\lambda \le 1-\frac{\sqrt{h\mu}}{2}$. We have the following equation:
    \begin{equation}
        \begin{aligned}
            \lambda^2 \left\|\begin{bmatrix} \sum_{i=0}^{n-k} \mulambda{1}^i\mulambda{2}^{n-k-i} & -\mulambda{1}\mulambda{2}\sum_{i=0}^{n-k-1} \mulambda{1}^i\mulambda{2} ^{n-k-1-i}  \end{bmatrix}\right\|^2 \le 4\lambda^2(n-k)^2 \alambda^{n-k}.
        \end{aligned}
    \end{equation}
    From the definition of $\y_i$ and Cauchy-Schwartz inequality, we have 
    \begin{equation}
        \|\y_i\|_{\A^2}^2\le 8\|\x_i\|_{\A^2}^2+2\|\x_{i-1}\|_{\A^2}^2\le8\|\z_i\|_\C^2+2\|\z_{i-1}\|_\C^2.
    \end{equation}

    Therefore, 
    \begin{equation}
        \begin{aligned}
            &\quad\E \left\|\sum_{k=0}^{n-1} \B^{k-i} \beps_i\right\|_\C^2 \\
            &\le 3\sum_{k=0}^{n-1}\sum_{i=1}^{d} 8\lambda_i^2(n-k)^2a_{\lambda_i}^{n-k}\cdot \|\y_k\|_{\A^2}^2\\
            &= 24\sum_{i=1}^d \sum_{k=0}^{n-1} \lambda_i^2(n-k)^2a_{\lambda_i}^{n-k}\cdot \|\y_k\|_{\A^2}^2\\
        \end{aligned}
        \label{equ:sumexponential}
    \end{equation}

    Then we calculate $\|\B^n \z_0\|_\C^2$. As $\x_{-1}=\x_0$, the contribution of an eigen-directions of $\A$ to the norm is 
    \begin{equation}
        \begin{aligned}
            \lambda^2 x_\lambda^2\left\|\B_\lambda^n \begin{bmatrix}
                1\\1
            \end{bmatrix} \right\|^2,
        \end{aligned}
    \end{equation}
    where $\lambda$ is the eigenvalue, and $x_\lambda$ is the coefficient of the eigen-decomposition of $\x_0$. By Lemma 19 of \cite{jin_accelerated_2017}, we have
    \begin{equation}
        \begin{aligned}
            \B_\lambda^n \begin{bmatrix}
                1\\1
            \end{bmatrix}
            &= \begin{bmatrix}
                \sum_{i=0}^n \mulambda{1}^i\mulambda{2}^{n-i} - \mulambda{1}\mulambda{2} \sum_{i=0}^{n-1} \mulambda{1}^i\mulambda{2}^{n-1-i}\\
                \sum_{i=0}^{n-1}\mulambda{1}^i\mulambda{2}^{n-1-i} - \mulambda{1}\mulambda{2} \sum_{i=0}^{n-2} \mulambda{1}^i\mulambda{2}^{n-2-i}
            \end{bmatrix}\\
            &= \frac{1}{2}\begin{bmatrix}
                \mulambda{1}^n+\mulambda{2}^n+(2-\mulambda{1}-\mulambda{2})\sum_{i=0}^n \mulambda{1}^i\mulambda{2}^{n-i}\\
                \mulambda{1}^{n-1}+\mulambda{2}^{n-1}+(2-\mulambda{1}-\mulambda{2})\sum_{i=0}^{n-1} \mulambda{1}^i\mulambda{2}^{n-1-i}
            \end{bmatrix}\\
            &=\frac{1}{2}\begin{bmatrix}
                \mulambda{1}^n+\mulambda{2}^n+(2-\mulambda{1}-\mulambda{2})\frac{\mulambda{1}^{n+1}-\mulambda{2}^{n+1}}{\mulambda{1}-\mulambda{2}}\\ &\\
                \mulambda{1}^{n-1}+\mulambda{2}^{n-1}+(2-\mulambda{1}-\mulambda{2})\frac{\mulambda{1}^{n}-\mulambda{2}^{n}}{\mulambda{1}-\mulambda{2}}
            \end{bmatrix}\\
            \label{equ:errorterm}
        \end{aligned}
    \end{equation}
    The $\frac{2-\mulambda{1}-\mulambda{2}}{\mulambda{1}-\mulambda{2}}$ term in \eqref{equ:errorterm} can be bounded as follows:
    \begin{equation}
        \begin{aligned}
            \frac{2-\mulambda{1}-\mulambda{2}}{\mulambda{1}-\mulambda{2}} &= \frac{2-(2-\beta)(1-h\lambda)}{\sqrt{(1-h\lambda)(h\lambda(2-\beta)^2-\beta^2)}}\\
            &\le \frac{\beta+h\lambda}{\sqrt{\frac14\cdot h\lambda}}\\
            &\le 2+\sqrt{h\lambda}\\
            &\le 3.
        \end{aligned}
    \end{equation}
    Therefore, 
    \begin{equation}
        \begin{aligned}
            &\quad\left\|\B_\lambda^n \begin{bmatrix}
                1\\1
            \end{bmatrix} \right\|^2 \\
            &\le \frac{1}{4}\cdot 4\left(|\mulambda{1}^{2n}| + |\mulambda{2}^{2n}| + 9|\mulambda{1}^{2n+2}| + 9|\mulambda{2}^{2n+2}| + |\mulambda{1}^{2n-2}| + |\mulambda{2}^{2n-2}| + 9|\mulambda{1}^{2n}| + 9|\mulambda{2}^{2n}| \right)\\
            &\le 40\left(1-\frac{\sqrt{h\mu}}{2}\right)^{2n-2},
        \end{aligned}
    \end{equation}
    and we have 
    \begin{equation}
        \|\B^n\z_0\|_\C^2 \le 40\left(1-\frac{\sqrt{h\mu}}{2} \right)^{2n-2}\|\z_0\|_\C^2.
    \end{equation}
    
    Finally we tackle the $\bepss_k$ terms. We have
    \begin{equation}
        \begin{aligned}
        &\quad\hntf(\x) - \tilde\nabla f(\x) \\
            &= \left( \frac{\tilde f_\delta (\x + \rho\gauss) - \tilde f_\delta(\x)}{\rho} - \left\langle\nabla f(\x), \gauss \right\rangle\right)\cdot\gauss\\
            &= \left( \left[\frac{\tilde f_\delta (\x + \rho\gauss)-\tilde f_\delta(\x)}{\rho} - \frac{f(\x+\rho\gauss) - f(\x)}{\rho} \right] + \left[\frac{f(\x+\rho\gauss) - f(\x)}{\rho} - \langle\nabla f(\x), \gauss\rangle\right]\right)\cdot\gauss.\\
            &=\left( \left[\frac{\tilde f_\delta (\x + \rho\gauss)-\tilde f_\delta(\x)}{\rho} - \frac{f(\x+\rho\gauss) - f(\x)}{\rho} \right] + \frac{\rho}{2}\|\gauss\|_\A^2\right)\cdot\gauss.\\
        \end{aligned}
    \end{equation}
    With the above equality, we have
    \begin{equation}
        \begin{aligned}
            &\quad\E \|\hntf(\x) - \tilde\nabla f(\x) \|_{\B^{k\top}\C\B^k}^2\\
            &\le 2\E \left[\frac{\tilde f_\delta (\x + \rho\gauss)-\tilde f_\delta(\x)}{\rho} - \frac{f(\x+\rho\gauss) - f(\x)}{\rho} \right]\cdot\left\|\begin{bmatrix}\gauss\\\mathbf 0\end{bmatrix}\right\|_{{\B^{k\top}\C\B^k}}^2 \\
            &\quad+ 2\E\frac{\rho}{2} \|\gauss\|_{\A}^2\cdot\left\|\begin{bmatrix}\gauss\\\mathbf 0\end{bmatrix}\right\|_{{\B^{k\top}\C\B^k}}^2\\
            &\le \frac{4\delta}{\rho} \E\left\|\begin{bmatrix}\gauss\\\mathbf 0\end{bmatrix}\right\|_{{\B^{kT}\C\B^k}}^2+\rho\E\|\gauss\|_{\A}^2\cdot\left\|\begin{bmatrix}\gauss\\\mathbf 0\end{bmatrix}\right\|_{{\B^{kT}\C\B^k}}^2\\
            &\le \frac{4\delta}{\rho} \tr\left({\B^{k\top}\C\B^k}\right)+3\rho\tr(\A)\cdot\tr({\B^{k\top}\C\B^k})\\
            &\le \left(\frac{4\delta}{\rho}+3\rho\tr(\A)\right) \cdot 4\sum_{i=1}^d \lambda_i^2 (k+1)^2 a_{\lambda_i}^k.
        \end{aligned}
    \end{equation}
    Therefore, 
    \begin{equation}
        \begin{aligned}
            \E \left\|\sum_{k=0}^{n-1} \B^{n-k-1} \bepss_k\right\|_\C^2 
            &\le n\sum_{k=0}^{n-1} \E_k\|\B^{n-k-1}\bepss_k\|_\C^2\\
            &\le n\sum_{k=1}^n \left(\frac{4\delta}{\rho} + 3\rho\tr(\A)\right)\cdot 4\sum_{i=1}^d\lambda_i^2(n-k)^2a_{\lambda_i}^{n-k}\\
            &= n\left(\frac{16\delta}{\rho}+12\rho\tr(\A)\right)\cdot \sum_{i=1}^d \sum_{k=1}^n \lambda_i^2(n-k)^2a_{\lambda_i}^{n-k}.
        \end{aligned}
    \end{equation}

    Finally, we use induction to prove that $\E \|\z_n\|_{\C}^2 < 200(1-b)^n\|\z_0\|_\C^2$ where $b = 1-\frac{\sqrt{h\mu}}{4}$ when $\frac{16\delta}{\rho}+12\rho\tr(\A)$ is small. Suppose that for $k<n$, we have $\E \|\z_k\|_{\C}^2 <200(1-b)^n \|\z_0\|_\C^2$. By \eqref{equ:znexpansion}, we have
    \begin{equation}
        \begin{aligned}
            \E\|\z_n\|_\C^2 &\le 120\left(1-\frac{\sqrt{h\mu}}{2}\right)^{2n-2}\|\z_0\|_\C^2 \\
            &\quad + 72h^2\sum_{i=1}^d \sum_{k=0}^{n-1} \lambda_i^2(n-k)^2 a_{\lambda_i}^{n-k}\cdot \|\y_k\|_{\A^2}^2\\
            &\quad+ 3nh^2 \left(\frac{16\delta}{\rho} +12\rho\tr(\A) \right)\cdot\sum_{i=1}^d\sum_{k=1}^n \lambda_i^2(n-k)^2a_{\lambda_i}^{n-k}.
        \end{aligned}
    \end{equation}
    By the definition of $\y_k$ and the assumption for induction, we have 
    \begin{equation}
        \E{\|\y_k\|_{\A^2}^2} \le 2000(1-b)^{n-1}\|\z_0\|_\C^2.
    \end{equation}
    Using the summation result:
    \begin{equation}
        \sum_{k=1}^n k^2 a^k <\frac{1}{(1-a)^3},
    \end{equation}
    we have
    \begin{equation}
    \begin{aligned}
        \E\|\z_n\|_\C^2  &\le 120\left(1-\frac{\sqrt{h\mu}}{2}\right)^{2n-2}\|\z_0\|_\C^2\\
        &\quad+ 144000(1-b)^{n-1}\sum_{i=1}^d \frac{h^2\lambda_i^2}{\left(1-\frac{a_{\lambda_i}}{b}\right)^3}\|\z_0\|_\C^2\\
        &\quad + 3n \left(\frac{16\delta}{\rho} +12\rho\tr(\A) \right)\cdot\sum_{i=1}^d \frac{h^2\lambda_i^2}{(1-a_{\lambda_i})^3}\\
        &\le 
        120\left(1-\frac{\sqrt{h\mu}}{2}\right)^{2n-2}\|\z_0\|_\C^2\\
        &\quad+576000(1-b)^{n-1}\sum_{i=1}^d \sqrt{h\lambda_i}\|\z_0\|_\C^2\\
        &\quad + 6n \left(\frac{16\delta}{\rho} +12\rho\tr(\A) \right)\cdot \sum_{i=1}^d \sqrt{h\lambda_i}.
    \end{aligned}
    \end{equation}
    By $h=\frac{1}{14400^2(\sum_i\lambda_i^{1/2})^2}$, we have 
    \begin{equation}
        \begin{aligned}
            \E\|\z_n\|_\C^2  \le 160(1-b)^{n-1}\|\z_0\|_\C^2 + 6n \left(\frac{16\delta}{\rho} +12\rho\tr(\A) \right)\cdot \sum_{i=1}^d \sqrt{h\lambda_i}.
        \end{aligned}
    \end{equation}
    Therefore, if $6n \left(\frac{16\delta}{\rho} +12\rho\tr(\A) \right)\ < 40 (1-b)^{n-1}\|\z_0\|_\C^2$, we have $\E \|\z_n\|_{\C}^2 < 200(1-b)^n\|\z_0\|_\C^2$. 

     Finally, we have
    \begin{equation}
    \begin{split}
        \|\z_n\|_\C^2  &= {\x_n}^\top \A^2\x_n +{\x_{n-1}}^\top \A^2\x_{n-1}\\ 
        &\ge \mu\left({\x_n}^\top \A\x_n + {\x_{n-1}}^\top \A\x_{n-1} \right)\\
        &= 2\mu(f(\x_n) + f(\x_{n-1})),
        \end{split}
    \end{equation}
    and 
    \begin{equation}
    \begin{split}
        \|\z_0\|_\C^2  &= 2{\x_0}^\top \A^2\x_n\\ 
        &\le 2L{\x_0}^\top \A\x^0\\
        &= 4Lf(\x_0).
        \end{split}
    \end{equation}
    Therefore,
    \begin{equation}
    \begin{split}
        \E f(\x_n) &\le \frac{1}{2\mu}\cdot 200(1-b)^n \cdot(4Lf(\x_0))\\
        &= 400\cdot\frac{L}{\mu}\cdot\left(1-\frac{\mu}{57600\sum_i \lambda_i^{1/2}} \right)^n\cdot f(\x_0).
    \end{split}
    \end{equation}
\end{proof}

\section{Proof of Theorem \ref{thm:gen-convex}}
In this section, we give the proof of Theorem \ref{thm:gen-convex}.
\subsection{Proof of Main Results}
We first present a theorem on the number of iterations of Algorithm \ref{alg:A-HPEzo}, whose proof can be found in \cite{Monteiro2013An}:
\begin{theorem}[Theorem 4.1 in \cite{Monteiro2013An}]
    If all the parameters satisfy the requirements of Algorithm \ref{alg:A-HPEzo}, then for every integer $1\le k\le n$, the following statements hold:
    \begin{equation}
        A_k\ge \left(\frac{2}{3}\right)^{7/2}\cdot\left(\frac{\sigma_l(1-\sigma^2)^{1/2}}{16DH}\right)\cdot k^{7/2},
    \end{equation}
    and
    \begin{equation}
        f(\y_k)-f^* \le \frac{3^{7/2}}{\sqrt{2}}\frac{HD^3}{\sigma_l\sqrt{1-\sigma^2}}\frac{1}{k^{7/2}}.
    \end{equation}
    \label{thm:MS4.1}
\end{theorem}

Now we give the proof of Theorem \ref{thm:gen-convex} below.
\setcounter{temp}{\value{theorem}}
\setcounter{theorem}{\value{thm:gen-convex}}
\begin{theorem}
    Assume the objective  function $f$ is convex and has $L$-continuous gradient and $H$-continuous Hessian matrices. Algorithm \ref{alg:A-HPEzo} needs
    \begin{equation}
        \tO\left(\frac{D\cdot\effdim_{1/2}}{\epsilon^{1/2}} + d\cdot D^{6/7}H^{2/7}\epsilon^{-2/7} \right)
    \end{equation}
    zeroth-order oracle calls to find an $\epsilon$-approximated solution with high probability.
\end{theorem}
\setcounter{theorem}{\value{temp}}

\begin{proof}
    Theorem \ref{thm:MS4.1} analyzes the outer loop of Algorithm \ref{alg:A-HPEzo}, so we only need to analyze the inner loop of Algorithm \ref{alg:A-HPEzo}, namely Algorithm \ref{alg:A-HPEzo-search} and Algorithm \ref{alg:approx}. For Algorithm \ref{alg:approx}, the zeroth-order oracle is called $\Theta(d)$ times. For Algorithm \ref{alg:A-HPEzo-search}, the problem \eqref{equ:APEsubproblem1} is solved $\cO\left(\left|\log\frac{\lambda_{j+1}}{\lambda_j}\right|\right)$ times.  
    Note that Theorem \ref{theorem:acc-zo} shows that solving \eqref{equ:APEsubproblem1} needs 
    \begin{equation}
        \cO\left(\left(\left(\frac{1}{\ltemp}\right)^{-1/2}\effdim_{1/2} + d\right)\cdot\log\frac{1}{\errorB}\cdot\log L\ltemp\right)
    \end{equation}
     zeroth-order oracle calls in expectation. By Markov's inequality, using same order of such oracles, we can find an approximated solution with a constant probability. So by repeating Algorithm \ref{alg:FG2} for logarithm times and taking the minimum solution, we can obtain a high probability result (see e.g. \citet{ghadimi2013stochastic}). Moreover,  we have $\ltemp\le \max\{\lambda_j, \lambda_{j+1}\}$.
    In order to find an $\epsilon$-approximated solution, we need to find the first $k$ such that $A_k\ge \frac{D^2}{\epsilon}$. Suppose that $A_k = \Theta\left(\frac{D^2}{\epsilon}\right)$, and in this case $k= \cO\left(D^{6/7}H^{2/7}\epsilon^{-2/7}\right)$. Therefore, ignoring all logarithmic factors, the zeroth-order oracle is called at most 
    \begin{equation}
    \begin{split}
        &\quad\sum_{j=1}^k \left( \tO\left(\left(\frac{1}{\ltemp}\right)^{-1/2}\effdim_{1/2} + d\right) + \Theta(d)\right)\\ &\le \sum_{j=1}^k \tO\left(\left(\frac{1}{\max\{\lambda_j, \lambda_{j+1}\}}\right)^{-1/2}\effdim_{1/2} + d\right)\\
        &= \tO\left(\effdim_{1/2}\cdot\sum_{j=1}^k \sqrt{\lambda_j} + kd \right)\\
        &\stackrel{\text{Lemma \ref{lem:MS3.7}}}{\le} \tO\left(\effdim_{1/2}\cdot\sqrt{A_k} + kd \right)\\
        &= \tO\left(\frac{D\cdot\effdim_{1/2}}{\epsilon^{1/2}} + d\cdot D^{6/7}H^{2/7}\epsilon^{-2/7} \right).
    \end{split}
    \end{equation}
\end{proof}

\subsection{Properties of Approximate Solutions}
In this subsection, we present a new framework for considering errors from inexactly solving   solutions. With Lemmas \ref{lem:HPEprecision} and \ref{lem:errorA}, we show that if $\errorA$ and $\errorB$ are small enough, the results in \cite{Monteiro2013An} still hold with a different numerical constant.

\begin{lemma}
    If $$\errorB < \frac{(\sigma-\sigma_u)^2}{2\lambda_{k+1}(L\lambda_{k+1} + 1 +(\sigma-\sigma_u)^2)\left(L+\frac{1}{\lambda_{k+1}} \right)} \cdot \left(f(\tx_k) - \min_y \left\{f_{\tx_k}(\y)+\frac{1}{2\lambda_{k+1}}\|\y-\tx_k\|^2\right\}\right),$$ $\y_{k+1}$ satisfies 
    \begin{equation}
        \|\lambda_{k+1}\nabla f(\y_{k+1})+ \y_{k+1}-\tx_{k}\|^2\le \sigma^2\|\y_{k+1}-\tx_k^2\|.
    \end{equation}
    \label{lem:HPEprecision}
\end{lemma}
\begin{proof}[Proof of Lemma \ref{lem:HPEprecision}]
    Denote 
    \begin{equation}
        g(\y) = f_{\tilde \x_k}(\y) + \frac{1}{2\lambda_{k+1}}\|\y-\tilde \x_k\|^2.
        \label{equ:gdef}
    \end{equation}
    By the $L+\frac{1}{\lambda_{k+1}}$-Lipschitz contiouity of $\nabla g$, we have
    \begin{equation}
        g(\y)-g^*\ge \frac{1}{2\left(L+\frac{1}{\lambda_{k+1}}\right)} \|\nabla g(\y)\|^2.
        \label{equ:errorBpf}
    \end{equation}
    Let $\y=\y_{k+1}$ in \eqref{equ:errorBpf}. We have
    \begin{equation}
        \begin{split}
        \|\lambda_{k+1}\nabla f_{\tx_k}(\y_{k+1})+ \y_{k+1}-\tx_{k}\|^2 &\stackrel{\eqref{equ:gdef}}{=} \lambda_{k+1}^2 \|\nabla g(\y)\|^2\\
        &\stackrel{\eqref{equ:errorBpf}}{\le} \left(2L\lambda_{k+1}^2 + 2\lambda_{k+1}\right)(g(\y_{k+1})-g^*)\\
        &\le \left(2L\lambda_{k+1}^2 + 2\lambda_{k+1}\right)\errorB.
        \end{split}
        \label{equ:errorBpf1}
    \end{equation}
    The optimal solution to \eqref{equ:APEsubproblem1} is 
    \begin{equation}
    \begin{split}
        \y^* &= \tx_{k} - \left(\nabla^2 f(\tx_k) + \frac{1}{\lambda_{k+1}}\I\right)^{-1} \nabla f(\tx_k)\\
    \end{split}
    \end{equation}
    and
    \begin{equation}
    \begin{split}
        g^* &= f(\tx_k) - \frac{1}{2}\left\langle \left(\nabla^2 f(\tx_k) + \frac{1}{\lambda_{k+1}}\I\right)^{-1} \nabla f(\tx_k), \nabla f(\tx_k)\right\rangle\\
        &\stackrel{}{\ge} f(\tx_k) - \frac{1}{2}\left(L+\frac{1}{\lambda_{k+1}}\right)\|\tx_k-\y^*\|^2\\
        &\stackrel{}{\ge} f(\tx_k) - \left(L+\frac{1}{\lambda_{k+1}}\right)\left(\|\tx_k-\y_{k+1}\|^2+\|\y_{k+1}-\y^*\|^2 \right)\\
        &\stackrel{a}{\ge} f(\tx_k) - \left(L+\frac{1}{\lambda_{k+1}}\right)\left(\|\tx_k-\y_{k+1}\|^2+ 2\lambda_{k+1}\errorB \right),
    \end{split}
    \label{equ:errorBpf2}
    \end{equation}
    where $\stackrel{a}{\ge}$ uses the $\lambda_{k+1}$-strong convexity of $g$.
    Therefore, if $\errorB < \frac{(\sigma-\sigma_u)^2}{2\lambda_{k+1}(L\lambda_{k+1} + 1 +(\sigma-\sigma_u)^2)\left(L+\frac{1}{\lambda_{k+1}} \right)}\cdot(f(\tx_k)-g^*)$, we have
    \begin{align}
            &\|\lambda_{k+1}\nabla f_{\tx_{k}}+\y_{k+1}-\tx_k\|^2 \notag\\
\stackrel{\eqref{equ:errorBpf1}}{\le}&\left(2L\lambda_{k+1}^2 + 2\lambda_{k+1}\right)\errorB\notag\\
            \le& \frac{(\sigma-\sigma_u)^2}{L+\frac{1}{\lambda_{k+1}}}(f(\tx_k)-g^*)\notag\\
            &\quad+ \left(2L\lambda_{k+1}^2+2\lambda_{k+1}-(2L\lambda_{k+1}^2 + 2\lambda_{k+1}+2(\sigma-\sigma_u)^2\lambda_{k+1})\right)\errorB\notag\\
    \stackrel{\eqref{equ:errorBpf2}}{=} &
            (\sigma-\sigma_u)^2 \|\y_{k+1}-\tx_k\|^2, 
        \label{equ:errorBpf3}
    \end{align}
    and we have
    \begin{align}
        &\quad \|\lambda_{k+1}\nabla f(\y_{k+1})+\y_{k+1}-\tx_k\|^2\\
        &=  \|(\lambda_{k+1}\nabla f_{\tx_k}(\y_{k+1})+\y_{k+1}-\tx_k)+(\lambda_{k+1}\nabla f_{\tx_k}(\y_{k+1})-\lambda_{k+1}\nabla f(\y_{k+1}))\|^2\notag\\
        &\le \|\lambda_{k+1}\nabla f_{\tx_k}(\y_{k+1})+\y_{k+1}-\tx_k\|^2 \notag\\
        &\quad + 2\|\lambda_{k+1}\nabla f(\y_{k+1})+\y_{k+1}-\tx_k\|\cdot\|\lambda_{k+1}\nabla f(\tx_k)\nabla^2 f(\tx_{k})(\y_{k+1}-\tx_k)-\lambda_{k+1}\nabla f(\y_{k+1}) \|\notag\\
        &\quad + \|\lambda_{k+1}\nabla f(\tx_k) + \nabla^2 f(\tx_{k})(\y_{k+1}-\tx_k)-\lambda_{k+1}\nabla f(\y_{k+1})\|^2\notag\\
        &\stackrel{\eqref{equ:errorBpf3}}{\le} (\sigma-\sigma_u)^2 \|\y_{k+1}-\tx_k\|^2 + 2(\sigma-\sigma_u) \|\y_{k+1}-\tx_k\|\cdot \frac{H\lambda_{k+1}\|\y_{k+1}-\tx_k\|}{2}\notag\\
        &\quad+ \left(\frac{H\lambda_{k+1}\|\y_{k+1}-\tx_k\|}{2}\right)^2 \notag\\
        &\le \left(\sigma - \sigma_u + \frac{H}{2}\cdot \frac{2\sigma_u}{H}\right)^2 = \sigma^2.\notag
    \end{align}
\end{proof}

\begin{lemma}
    If $\errorA<\frac{D}{N^{3/2}}$, then we have
    \begin{equation}
        \sum_{j=1}^k \|\x_j-\x_j^*\|^2\le D^2.
        \label{equ:errorA}
    \end{equation}
    for $k\le N$.
    \label{lem:errorA}
\end{lemma}
\begin{proof}[Proof of Lemma \ref{lem:errorA}]
    By the definition of $\x_j$, $\x_j^*$ and $A_k$, we have
    \begin{equation}
    \begin{split}
        \sum_{j=1}^k \|\x_j-\x_j^*\|^2 &= \sum_{j=1}^k \left\|\sum_{i=1}^j a_i(v_i-\nabla f(\y_i))\right\|^2\\
        &\le \sum_{j=1}^k \left(\sum_{i=1}^j a_i\cdot \frac{\errorA}{a_i}\right)^2\\
        &=\sum_{j=1}^k j^2\errorA^2\le k^3\errorA^2.
    \end{split}
    \end{equation}
    Therefore, if $\errorA \le \frac{D}{N^{3/2}}$, we have \eqref{equ:errorA}.
\end{proof}

In order to analyze $f(\x_k)$, we define the affine maps $\gamma_k$ as
\begin{equation}
    \gamma_k(\x) = f(\y_k) + \langle\nabla f(\y_k), \x-\y_k\rangle.
    \label{equ:gammakdef}
\end{equation}
and the aggregate affine maps $\Gamma_k$ recursively as:
\begin{equation}
    \Gamma_0\equiv 0, \qquad\Gamma_{k+1} = \frac{A_k}{A_{k+1}}\Gamma_k + \frac{a_{k+1}}{A_{k+1}}\gamma_{k+1}. 
    \label{equ:Gammadef}
\end{equation}
We define
\begin{equation}
    \x_k^* = \x_0 - \sum_{j=1}^k a_{j}\nabla f(\y_{k+1}),
\end{equation}
\begin{lemma}[Lemma 3.2 of \cite{Monteiro2013An}]
For every integer $k\ge 0$, there hold:
\begin{enumerate}
    \item $\gamma_{k+1}$ is affine and $\gamma_{k+1}\le f$.
    \item $\Gamma_k$ is affine and $A_k\Gamma_k\le A_k f$.
    \item $\x_k^* = \argmin_\x A_k\Gamma_k(\x) + \frac{1}{2}\|\x-\x_0\|^2$.
\end{enumerate}
\label{lem:MS3.2}
\end{lemma}


\begin{lemma}[Inspired by Lemma 3.4 of \cite{Monteiro2013An}]
    For integer $k\ge 0$, define 
    \begin{equation}
        \beta_k = \left(\inf_{\x\in \R^n} A_k\Gamma_k(\x) + \frac{1}{2}\|\x-\x_0\|^2 \right)-A_kf(\y_k).
        \label{equ:betakdef}
    \end{equation}
    If $ \|\lambda\bv+\y-\x\|^2\le \sigma^2\|\y-\x\|^2$, we have $\beta_0=0$, and 
    \begin{equation}
        \beta_{k+1}\ge \beta_k + \frac{(1-\sigma^2)A_{k+1}}{2\lambda_{k+1}} \|\y_{k+1}-\tx_k\|^2 -\frac{1}{2}\|\x_k-\x_k^*\|^2.
        \label{equ:MS3.4}
    \end{equation}
    \label{lem:MS3.4}
\end{lemma}

\begin{proof}[Proof of Lemma \ref{lem:MS3.4}]
    We have $\beta_0 = 0$ since $A_0=0$. For $\x\in \R^n$, define 
    \begin{equation}
        \tx = \frac{A_k}{A_{k+1}}\y_k+\frac{a_{k+1}}{A_{k+1}}\x.
        \label{equ:txdef}
    \end{equation}
    By the definition of $\tx_k$ in Algorithm \ref{alg:A-HPEzo} and the affinity of $\gamma$, we have
    \begin{align}
        \tx-\tx_k&=\frac{a_{k+1}}{A_{k+1}}(\x-\x_k)\label{equ:APEzo1},\\
        \gamma_{k+1}(\tx) &= \frac{A_k}{A_{k+1}}\gamma_{k+1}(\y_k) + \frac{a_{k+1}}{A_{k+1}}\gamma_{k+1}(\x).
    \end{align}
    We have the following equality:
    \begin{equation}
        \begin{split}
           & \!\!\!\!\!\!\!\!\!\!\!\!A_{k+1}\Gamma_{k+1}(\x) + \frac{1}{2}\|\x-\x_0\|^2\\ &\stackrel{\eqref{equ:Gammadef}}{=} a_{k+1}\gamma_{k+1}(\x) + A_k\Gamma_k(\x)+\frac{1}{2}\|\x-\x_0\|^2 \\
            &\stackrel{\text{Lemma \ref{lem:MS3.2} and \eqref{equ:betakdef}}}{=}a_{k+1}\gamma_{k+1}(\x) + A_k f(\y_k)+\beta_k+\frac{1}{2}\|\x-\x_k^*\|^2 \\
            &\ge a_{k+1}\gamma_{k+1}(\x) + A_k f(\y_k)+\beta_k+\frac{1}{4}\|\x-\x_k\|^2 -\frac{1}{2}\|\x_k-\x_k^*\|^2\\
            &\stackrel{\text{Lemma \ref{lem:MS3.2}}}{\ge} a_{k+1}\gamma_{k+1}(\x) + A_k \gamma_{k+1}(\y_k)+\beta_k+\frac{1}{4}\|\x-\x_k\|^2 -\frac{1}{2}\|\x_k^2-\x_k^*\|^2\\
            &\stackrel{\eqref{equ:txdef}}{\ge} A_{k+1}\gamma_{k+1}(\tx) +\beta_k+\frac{1}{4}\|\x-\x_k\|^2 -\frac{1}{2}\|\x_k-\x_k^*\|^2\\
            &\stackrel{\eqref{equ:txdef}}{\ge} A_{k+1}\gamma_{k+1}(\tx) +\beta_k+\frac{A_{k+1}^2}{4a_{k+1}^2}\|\tx-\tx_k\|^2 -\frac{1}{2}\|\x_k-\x_k^*\|^2\\
            &\stackrel{\text{Lemma \ref{lem:MS3.1}}}{\ge} \beta_k+ A_{k+1}\left(\gamma_{k+1}(\tx) + \frac{1}{4\lambda_{k+1}}\|\tx-\tx_k\|^2\right) -\frac{1}{2}\|\x_k-\x_k^*\|^2.
        \end{split}
        \label{equ:MS3.4-1}
    \end{equation}
    With \eqref{equ:MS3.4-1}, we have
    \begin{equation}
        \begin{split}
            &\quad\beta_{k+1} + A_{k+1}f(\y_{k+1})\\
            &\stackrel{\eqref{equ:betakdef}}{=} \inf_\x \left\{A_{k+1}\Gamma_{k+1} + \frac{1}{2}\|\x-\x_0\|^2 \right\}
            \\&\stackrel{\eqref{equ:MS3.4-1}}{\ge} \beta_k + A_{k+1}\inf_{\tx}\left\{\gamma_{k+1}(\tx) + \frac{1}{2\lambda_{k+1}}\|\tx-\tx_k\|^2 \right\} - \frac{1}{2}\|\x_k-\x_k^*\|^2\\
            &\stackrel{\text{\eqref{equ:gammakdef}}}{\ge} \beta_k + A_{k+1}f(\y_{k+1}) +  A_{k+1}\inf_{\tx}\left\{ \langle\nabla f(\y_{k}), \tx-\y_{k+1}\rangle + \frac{1}{4\lambda_{k+1}}\|\tx-\tx_k\|^2 \right\} - \frac{1}{2}\|\x_k-\x_k^*\|^2\\
            &\stackrel{\text{Lemma \ref{lem:MS3.3}}}{\ge}\beta_k + A_{k+1}f(\y_{k+1}) + \frac{(1-\sigma^2)A_{k+1}}{4\lambda_{k+1}}\|\y_j-\tx_{j-1}\|^2 - \frac{1}{2}\|\x_k-\x_k^*\|^2\\
        \end{split}
    \end{equation}
    Therefore, we have \eqref{equ:MS3.4}.
\end{proof}
\begin{lemma}
    Let $D = \|\x_0-\x^*\|$. If 
    \begin{equation}
        \sum_{j=1}^k \|\x_j-\x_j^*\|^2\le D^2,
        \label{equ:accuracyassumption}
    \end{equation}
    then for every integer $k\ge 1$, 
    \begin{equation}
        \frac{1}{4}\|\x_k-\x^*\|^2+A_k[f(\y_k)-f^*] + \frac{1-\sigma^2}{4}\sum_{j=1}^k \frac{A_j}{\lambda_j}\|\y_j-\tx_{j-1}\|^2\le D^2.
    \end{equation}
    As a consequence, 
    \begin{equation}
        f(\y_k)-f^*\le \frac{D^2}{A_k},\qquad \|\x_k-\x^*\|\le 2D,
    \end{equation}
    and if $\sigma^2\le 1$,
    \begin{equation}
        \sum_{j=1}^k \frac{A_k}{\lambda_j}\|\y_j-\x_{j-1}\|^2\le \frac{4D^2}{1-\sigma^2}.
    \end{equation}
    
    \label{lem:MS3.6}
\end{lemma}

\begin{proof}[Proof of Lemma \ref{lem:MS3.6}]
    Summing \eqref{equ:MS3.4} from $k=0$ to $k-1$, we have 
    \begin{equation}
        \beta_k\ge \frac{1-\sigma^2}{2}\sum_{j=1}^k \frac{A_{k+1}}{\lambda_{k+1}} \|\y_{k+1}-\tx_k\|^2 - \frac{1}{2}\sum_{j=1}^{k-1} \|\x_k-\x_k^*\|^2.
    \end{equation}
    Using the definition of $\beta_k$ in \eqref{equ:betakdef}, we have
    \begin{align}
        &A_k f(\y_k) + \frac{1-\sigma^2}{4} \sum_{j=1}^k \frac{A_j}{\lambda_j} \|\y_j-\tx_{j-1}\|^2  \label{equ:MS3.6-1}\\\le& \inf_{\x\in \R^n}\left(A_k\Gamma_k+\frac{1}{2}\|\x-\x_0\|^2 \right) + \frac{1}{2}\sum_{j=1}^{k-1} \|\x_j-\x_j^*\|^2.\notag
    \end{align}

    With Lemma \ref{lem:MS3.2}, we have
    \begin{equation}
        \inf_{\x\in\R^n} \left(A_k\Gamma_k(\x) + \frac{1}{2}\|\x-\x_0\|^2\right) + \frac{1}{2}\|\x-\x_k^*\|^2 = A_k\Gamma_k(\x) + \frac{1}{2}\|\x-\x_0\|^2.
        \label{equ:MS3.6-2}
    \end{equation}
    Plugging \eqref{equ:MS3.6-2} into \eqref{equ:MS3.6-1}, we have
    \begin{equation}
        \begin{split}
             &\qquad A_k f(\y_k) + \frac{1-\sigma^2}{4} \sum_{j=1}^k \frac{A_j}{\lambda_j} \|\y_j-\tx_{j-1}\|^2 + \frac{1}{4}\|\x-\x_k\|^2\\ 
             &\le A_k\Gamma_k(\x) + \frac{1}{4}\|\x-\x_k\|^2-\frac{1}{2} \|\x-\x_k^*\|^2+\frac{1}{2}\|\x-\x_0\|^2+ \frac{1}{2}\sum_{j=1}^{k-1} \|\x_j-\x_j^*\|^2\\
             &\le A_k\Gamma_k(\x) + \frac{1}{2}\|\x-\x_0\|^2+ \frac{1}{2}\sum_{j=1}^{k} \|\x_j-\x_j^*\|^2\\
        \end{split}
        \label{equ:MS3.6-3}
    \end{equation}
    Letting $\x=\x^*$ in \eqref{equ:MS3.6-3}, we have
    \begin{equation}
        \begin{split}
            &\quad A_k f(\y_k) + \frac{1-\sigma^2}{4} \sum_{j=1}^k \frac{A_j}{\lambda_j} \|\y_j-\tx_{j-1}\|^2 + \frac{1}{4}\|\x^*-\x_k\|^2 \\
            &\le A_k\Gamma_k(\x^*) + \frac{1}{2}\|\x^*-\x_0\|^2+ \frac{1}{2}\sum_{j=1}^{k} \|\x_j-\x_j^*\|^2\\
            &\le A_k f^* + \frac{1}{2}\|\x^*-\x_0\|^2+ \frac{1}{2}\sum_{j=1}^{k} \|\x_j-\x_j^*\|^2.
        \end{split}
    \end{equation}
    Therefore, using Lemma \ref{lem:MS3.2} and \eqref{equ:accuracyassumption}, we have
    \begin{equation}
        A_k (f(\y_k)-f^*) + \frac{1-\sigma^2}{4} \sum_{j=1}^k \frac{A_j}{\lambda_j} \|\y_j-\tx_{j-1}\|^2 + \frac{1}{4}\|\x^*-\x_k\|^2 \le D^2
    \end{equation}
\end{proof}

\subsection{Useful Lemmas in \cite{Monteiro2013An}}
In this subsection, we list some results leading to Theorem \ref{thm:MS4.1} in \cite{Monteiro2013An}. 

\begin{lemma}[Lemma 3.1 of \cite{Monteiro2013An}]
    \begin{equation}
        \lambda_{k+1}A_{k+1} = a_{k+1}^2.
    \end{equation}
    \label{lem:MS3.1}
\end{lemma}

\begin{lemma}[Lemma 3.3 of \cite{Monteiro2013An}]
    The inequality
    \begin{equation}
        \|\lambda\bv+\y-\x\|^2\le \sigma^2\|\y-\x\|^2
    \end{equation}
    is equivalent to inequality
    \begin{equation}
        \min_{\z\in \R^n} \left\{\langle\bv,\z-\y \rangle+\frac{1}{2\lambda}\|\z-\x\|^2 \right\} \ge \frac{1-\sigma^2}{2\lambda}\|\y-\x\|^2.
    \end{equation}
    \label{lem:MS3.3}
\end{lemma}

\begin{lemma}[Lemma 3.7 of \cite{Monteiro2013An}]
    For every integer $k\ge 0$, 
    \begin{equation}
        \sqrt{A_{k+1}} \ge \sqrt{A_k} + \frac{1}{2} \sqrt{\lambda_{k+1}}.
    \end{equation}
    \label{lem:MS3.7}
\end{lemma}

\begin{lemma}[Lemma 4.2 of \cite{Monteiro2013An}]
    If all the parameters satisfy the requirements of Algorithm \ref{alg:A-HPEzo}, then for every integer $1\le k\le N$, 
    \begin{equation}
        \sum_{j=1}^k \frac{A_j}{\lambda_j^3} \le \frac{H^2D^2}{\sigma_l^2(1-\sigma^2)}.
    \end{equation}
    \label{lem:MS4.2}
\end{lemma}

\begin{lemma}[Lemma 4.4 of \cite{Monteiro2013An}]
    If all the parameters satisfy the requirements of Algorithm \ref{alg:A-HPEzo}, then for $1\le k\le N$,
    \begin{equation}
        A_k\ge \frac{1}{4}w \left(\sum_{j=1}^k A_j^{1/3} \right)^{7/3},
    \end{equation}
    where 
    \begin{equation}
        w = \frac{\sigma_l^2(1-\sigma^2)}{4H^2D^2}.
    \end{equation}
    \label{lem:MS4.4}
\end{lemma}


\section{Proof of Theorem \ref{thm:gen-nonconvex}}
In this section, we provide the proof of Theorem \ref{thm:gen-nonconvex}. 
\subsection{Properties of Approximated Solutions}
We define 
\begin{equation}
    \tx_{k+1} = \argmin_\y f(\x_k) + \langle\nabla f(\x_k), \y-\x_k\rangle + \frac12\langle\nabla^2 f(\x_k) (\y-\x_k), \y-\x_k\rangle + \frac{H}{6}\|\y-\x_k\|^3,
    \label{equ:Cubictxdef}
\end{equation}
which is the exact solution of the cubic regularization subproblem, and 
\begin{equation}
     \tilde r_{k+1} = \|\tx_{k+1}-\x_k\|.
\end{equation}
We first present some results which considers the error of inexact solutions.

\begin{lemma}
    If $\errorC < \min\left\{\frac{\epsilon}{800\left(\frac{16\cdot\left( 24\Delta/H\right)^{1/3}}{\sqrt{\epsilon H}} + 1\right)}, \frac{\epsilon}{800}, \frac{1}{2000}\left(\frac{\epsilon}{H} \right)^{3/2}\right\}$ and \\$\errorD < \min\left\{\frac{\sqrt{\frac{\epsilon}{H}}}{200\left(\frac{16\cdot\left( 24\Delta/H\right)^{1/3}}{\sqrt{\epsilon H}} + 1\right)}, \sqrt{\frac{\epsilon}{40000H}}\right\}$, then we have
    \begin{equation}
        \|\tx_k-\x_k\| < \sqrt\frac{\epsilon}{10000H}.
    \end{equation}
    \label{lem:Cubicsearchaccuracy}
\end{lemma}
\begin{proof}[Proof of Lemma \ref{lem:Cubicsearchaccuracy}]
     For any $r\ge0$, Define
    \begin{equation}
        g_r(\y) = f(\x_k) + \langle\nabla f(\x_k), \y-\x_k\rangle + \frac12\langle\nabla^2 f(\x_k) (\y-\x_k), \y-\x_k\rangle + \frac{rH}{2}\|\y-\x_k\|^2.
    \end{equation}
    We first show that 
    \begin{equation}
        \tx_{k+1} = \argmin_\y g_{\tilde r_{k+1}}(\y).
        \label{equ:txprop}
    \end{equation}
    Indeed, according to Lemma \ref{lem:Cubic1}, $g_{\tilde r_{k+1}}$ is $\frac{\tilde r_{k+1}H}{2}$-strongly convex, and according to \eqref{equ:Cubictxdef}, we have $\nabla g_{\tilde r_{k+1}}(\tx_{k+1}) = \mathbf 0$. Thus we have \eqref{equ:txprop}. 

    By the definition of $r_{k+1}$ and $r_u$, we have 
    \begin{equation}
        \tilde r_{k+1} + 4\frac{\errorC}{\sqrt{\epsilon H}}+\errorD \ge r_l+\errorD\ge r_u\ge \tilde r_{k+1} - 4\frac{\errorC}{r_{k+1}H}\ge \tilde r_{k+1} - 4\frac{\errorC}{\sqrt{\epsilon H}},
        \label{equ:rerror1}
    \end{equation}
    and 
    \begin{equation}
        \tilde r_{k+1} + 4\frac{\errorC}{\sqrt{\epsilon H}}\ge r_l\ge \tilde r_u - \errorD \ge \tilde r_{k+1} - 4\frac{\errorC}{\sqrt{\epsilon H}}-\errorD.
        \label{equ:rerror2}
    \end{equation}
    Therefore, 
    \begin{equation}
    \begin{split}
        \|\x_{k+1}-\tx_{k+1}\|&\le \frac{8\left(4\frac{\errorC}{\sqrt{\epsilon H}}+\errorD\right)\|\tx_{k+1}-\x_k\|}{\tilde r_{k+1}H} + \left(4\frac{\errorC}{\sqrt{\epsilon H}}+\errorD\right)\\
        &\le \frac{16\left(4\frac{\errorC}{\sqrt{\epsilon H}}+\errorD\right)(24\Delta/H)^{1/3}}{\sqrt{\epsilon H}} + \left(4\frac{\errorC}{\sqrt{\epsilon H}}+\errorD\right) \\
    \end{split}
    \end{equation}
    Therefore, it can be verified that if the assumptions of Lemma \ref{lem:Cubicsearchaccuracy} are satisfied, then $\|\tx_k-\x_k\| < \sqrt\frac{\epsilon}{H}$.
\end{proof}

\begin{lemma}
    Define 
    \begin{equation}
        \tilde f_\x(\y) = f(\x) + \langle\nabla f(\x), \y-\x\rangle + \langle \nabla^2 f(\x)(\y-\x),\y-\x\rangle + \frac{H}{6}\|\y-\x\|^3.
    \end{equation}
    Then we have
    \begin{equation}
        \tilde f_{\x_k}(\x_{k+1})-\tilde f_{\x_k}(\tx_{k+1})\le \frac{H\tilde r_{k+1}^3}{500}.
    \end{equation}
    \label{lem:Cubicvalue}
\end{lemma}
\begin{proof}[Proof of Lemma \ref{lem:Cubicvalue}]
    \begin{equation}
        \begin{split}
            \tilde f_{x_k} (\x_{k+1}) &= g_{r_{k+1}} (\tx_{k+1}) + \frac{H}{6}\|\x_{k+1}-\x_k\|^3-\frac{Hr_{k+1}}{2}\|\x_{k+1}-\x_k\|^2\\
            &\le g_{r_{k+1}}(\x_{k+1}^*) + \errorC + \frac{H}{6}\|\x_{k+1}-\x_k\|^3-\frac{Hr_{k+1}}{2}\|\x_{k+1}-\x_k\|^2\\
            &\le g_{\tilde r_{k+1}}(\tx_{k+1}) + \frac{H(r_{k+1}-\tilde r_{k+1})}{2}\|\tx_{k+1}-\x_k\|^2 + \errorC\\
            &\quad + \frac{H}{6}\|\x_{k+1}-\x_k\|^3-\frac{Hr_{k+1}}{2}\|\x_{k+1}-\x_k\|^2\\
            &= \tilde f_{\x_k}(\tx_{k+1}) + \frac{H\tilde r_{k+1}^3}{3} +\frac{H(r_{k+1}-\tilde r_{k+1})\tilde r_{k+1}^2}{2} + \errorC \\
            &\quad + \frac{H}{6}\|\x_{k+1}-\x_k\|^3-\frac{Hr_{k+1}}{2}\|\x_{k+1}-\x_k\|^2.
        \end{split}
    \end{equation}
    By \eqref{equ:rerror1} and \eqref{equ:rerror2} and the definition of $\errorC$ and $\errorD$, we have
    \begin{equation}
        \frac{100}{101}\tilde r_{k+1} \le\|\x_{k+1}-\x_k\|< \frac{100}{99}\tilde r_{k+1},\quad r_{k+1}-\tilde r_{k+1}\le \frac{1}{99}\tilde r_{k+1}.
    \end{equation}
    Therefore, 
    \begin{equation}
        \tilde f_{\x_k}(\x_{k+1})-\tilde f_{\x_k}(\tx_{k+1})\le \frac{H\tilde r_{k+1}^3}{500}.
    \end{equation}
\end{proof}

\begin{theorem}
    We have
    \begin{equation}
        \|\nabla f(\x_k)\| \le H\tilde r_{k}^2 + \frac{\epsilon}{100}, \quad \nabla^2 f(\x_k)\succeq - \left(\frac{H\tilde r_k}{2} + \frac{\sqrt{H\epsilon}}{100}\right)\I
    \end{equation}
    \label{thm:cubicstationary}
\end{theorem}
\begin{proof}[Proof of Theorem \ref{thm:cubicstationary}]
    The results of $\nabla^2 f(\tx_k)$ follows from Lemma \ref{lem:Cubic1}, Lemma \ref{lem:Cubic3}, and Lemma \ref{lem:Cubicsearchaccuracy}. 

    By the $(L+Hr_{k+1})$-Lipschitz continuity of $\nabla f$
    \begin{equation}
        \|\nabla f(\x_{k}) + \nabla^2 f(\x_{k})(\x_{k+1}-\x_k) + Hr_{k+1}(\x_{k+1}-\x)\| \le \sqrt{2(L+Hr_{k+1})\errorC},
    \end{equation}
    and
    \begin{equation}
        \|\nabla f(\x_{k+1})-\nabla f(\x_k)-\nabla ^2 f(\x_k)(\x_{k+1}-\x_k)\|\le \frac{H}{2}\|\x_{k+1}-\x_k\|^2.
    \end{equation}
    Therefore, we have
    \begin{equation}
    \begin{split}
        \|\nabla f(\x_{k+1})\| &\le Hr_{k+1}\|\x_{k+1}-\x_k\| + \frac{H}{2}\|\x_{k+1}-\x_{k}\|^2+\sqrt{2(L+Hr_{k+1})\errorC}\\
        &\le 2H\tilde r_{k+1}^2+\frac{\epsilon}{100}.
    \end{split}
    \end{equation}
\end{proof}

\begin{theorem}[Theorem 1 of \cite{nesterov_accelerating_2007}]
    Define $\Delta = f(\x^0)-f^*$. If $$\errorC < \min\left\{\frac{\epsilon}{800\left(\frac{16\cdot\left( 24\Delta/H\right)^{1/3}}{\sqrt{\epsilon H}} + 1\right)}, \frac{\epsilon}{800}, \frac{1}{2000}\left(\frac{\epsilon}{H} \right)^{3/2}\right\}$$ and $\errorD < \min\left\{\frac{\sqrt{\frac{\epsilon}{H}}}{200\left(\frac{16\cdot\left( 24\Delta/H\right)^{1/3}}{\sqrt{\epsilon H}} + 1\right)}, \sqrt{\frac{\epsilon}{40000H}}\right\}$, 
    \begin{equation}
        \sum_{i=0}^\infty \|\tx_{i+1}-\x_i\|^3\le \frac{24\Delta}{H}.
    \end{equation}
    \label{thm:Cubic1}
\end{theorem}

\begin{proof}[Proof of Theorem \ref{thm:Cubic1}]
    By Lemma \ref{lem:Cubic4} and Lemma \ref{lem:Cubicvalue}, we have
    \begin{equation}
        \begin{split}
            f(\x_k)-f(\x_{k+1})&\ge  f(\tx_k)- \tilde f_{\x_k}(\x_{k+1})\\
            &= f(\tx_k) - \tilde f_{\x_k}(\tx_{k+1}) + \tilde f_{\x_k}(\tx_{k+1}) - \tilde f_{\x_k}(\x_{k+1})\\
            &\ge \frac{H\tilde r_{k+1}^3}{12}-\frac{H\tilde r_{k+1}^3}{500} \ge \frac{H\tilde r_{k+1}^3}{24}.
        \end{split}
        \label{equ:Cubic1pf}
    \end{equation}
    Summing \eqref{equ:Cubic1pf} from $k=0$ to $\infty$ yields the desired result.
\end{proof}

\subsection{Proof of Main Results}

By Theorems \ref{thm:Cubic1} and  \ref{thm:cubicstationary}, we have the following theorem:
\begin{theorem}
    Algorithm \ref{alg:cubiczo} finds an $(\epsilon, \sqrt{H\epsilon})$-stationary point in $\frac{48\sqrt{2}H^{1/2}\Delta}{\epsilon^{3/2}}$ rounds.
\end{theorem}

\setcounter{temp}{\value{theorem}}
\setcounter{theorem}{\value{thm:gen-nonconvex}}
\begin{theorem}
    Assume the objective  function $f$ is convex and has $L$-continuous gradients and $H$-continuous Hessian matrices. Algorithm \ref{alg:A-HPEzo} finds an $(\epsilon, \sqrt{H\epsilon})$-SSP of $f$ in 
    \begin{equation}
        \tO\left(\efftrace_{1/2}H^{1/4}\Delta \epsilon^{-7/4} + d H^{1/2}\Delta\epsilon^{-3/2}\right)
    \end{equation}
    zeroth-order oracle calls with high probability.
\end{theorem}
\setcounter{theorem}{\value{temp}}

\begin{proof}
    We analyze the inner loop of Algorithm \ref{alg:cubiczo}, namely Algorithm \ref{alg:ZCubic-search}. In Algorithm \ref{alg:cubiczo}, problem \eqref{equ:Cubicsubproblem1} is solved 
    \begin{equation}
        \cO\left(\left|\log\frac{r_k}{r_{k+1}}\right|+ \max\left\{1,\log\frac{r_{k+1}}{\errorD}  \right\} \right)
        \label{equ:cubiccalls1}
    \end{equation}
    times. By Theorem \ref{theorem:acc-zo}, solving subproblem \eqref{equ:Cubicsubproblem1} needs
    \begin{equation}
        \cO\left(\left((H\rtemp)^{-1/2}\efftrace_{1/2}(f) + d\right) \cdot\log\frac{1}{\errorC}\cdot\log{\frac{L}{H\rtemp}} \right)
        \label{equ:cubiccalls2}
    \end{equation}
    calls to the zeroth-order oracle in average. By Markov's inequality, using same order of such oracles, we can find an approximated solution with a constant probability. So by repeating Algorithm \ref{alg:FG2} for logarithm times and taking the minimum solution. So by repeating Algorithm \ref{alg:FG2} for logarithm times, we can obtain a high probability result (see e.g. \citet{ghadimi2013stochastic}). The maximum calls in solving one problem depends on the smallest possible value of $\rtemp$. It can be verified that $\rtemp \ge \min\left\{r_k, \frac{r_{k+1}}{2} \right\}$ in Algorithm \ref{alg:ZCubic-search}. Assume without loss of generality that $r_k\ge \sqrt{\frac{\epsilon}{2H}}$ for $k<N$, and $r_N < \sqrt{\frac{\epsilon}{2H}}$ where $N = \cO\left( H^{1/2}\Delta\epsilon^{-3/2}\right)$. By the definition of $r_N$, we have $r_N\ge \errorD= \Omega\left( \sqrt{\frac{\epsilon}{H}}\right)$. The logarithm factors in \eqref{equ:cubiccalls1} and \eqref{equ:cubiccalls2} can be bounded by linear combinations of $\log \frac{1}{\epsilon}$, $\log\Delta$, $\log L$ and $\log H$. Therefore, ignoring all logarithmic factors, the zeroth-order oracle is called at most
    \begin{equation}
        \begin{split}
        &\quad \sum_{j=1}^k \tO\left((H\min\{r_j,r_{j-1}\})^{-1/2}\efftrace_{1/2}(f) + d \right)\\
        &\le \sum_{j=1}^{k} \tO\left((\sqrt{H\epsilon})^{-1/2}\efftrace_{1/2}(f) + d \right)\\
        &= \tO\left(\efftrace_{1/2}(f)H^{1/4}\Delta \epsilon^{-7/4} + d H^{1/2}\Delta\epsilon^{-3/2}\right).
        \end{split}
    \end{equation}
\end{proof}

\subsection{Useful Results in \cite{nesterov_cubic_2006}}
In this subsection, we present some  results in \cite{nesterov_cubic_2006}, which we is used in our analysis.

\begin{lemma}[Proposition 1 of \cite{nesterov_cubic_2006}]
    \begin{equation}
        \nabla^2 f(\x) + \frac{1}{2}H\|\tx_{k+1}-\x_k\|\I\succeq \mathbf 0.
    \end{equation} 
    \label{lem:Cubic1}
\end{lemma}

\begin{lemma}[Lemma 2 of \cite{nesterov_cubic_2006}]
    For any $k\ge 0$, we have
    \begin{equation}
        \langle\nabla f(\x_k), \x_k-\tx_{k+1}\rangle \ge 0. 
    \end{equation}
    \label{lem:Cubic2}
\end{lemma}

\begin{lemma}[Lemma 3 of \cite{nesterov_cubic_2006}]
     For any $k\ge 0$, we have
    \begin{equation}
        \|\nabla f(\tx_{k+1})\|\le H\|\tx_{k+1}-\x_k\|^2.
    \end{equation}
    \label{lem:Cubic3}
\end{lemma}

\begin{lemma}[Lemma 4 of \cite{nesterov_cubic_2006}]
    \begin{equation}
        f(\x_k)-f(\tx_{k+1})\ge \frac{H}{12}\|\tx_{k+1}-\x_k\|^3.
    \end{equation}
    \label{lem:Cubic4}
\end{lemma}

\section{Proof of Additional Lemmas}
\setcounter{temp}{\value{theorem}}
\setcounter{theorem}{\value{pro:poly}}
\begin{proposition}
    Assume  for any $\x$ and $\alpha>0$, there exists constant $C>0$ and $\beta>0 $ such that $ \sigma_i(\nabla^2 f(\x)) \leq  \frac{C}{i^\beta}$ for $i\in[d]$, then we have
    \begin{equation}
        \effdim_{\alpha}\leq
\begin{cases}
\frac{2^{\alpha\beta-1}C^\alpha}{\alpha\beta-1}, \quad &\alpha\beta >1,  \quad \textit{dimensional free},\\
C^\alpha\log(2d+1), \quad &\alpha\beta =1,  \quad \textit{logarithmic growth on } d,\\
\frac{C^\alpha}{1-\alpha\beta}(d+1)^{1-\alpha \beta}, \quad &\alpha\beta <1, \quad  \textit{improve by a } \Theta\left(d^{\alpha\beta}\right) \textit{factor}. 
\end{cases}        
    \end{equation}
\end{proposition}
\setcounter{theorem}{\value{temp}}

\begin{proof}
    \begin{equation}
        \begin{split}
            \effdim_\alpha &= \sup_{\x\in\R^d}\sum_{i=1}^d \sigma_i^\alpha(\nabla^2 f(\x)) \le \sum_{i=1}^d \left(\frac{C}{i^\beta}\right)^\alpha = C^\alpha\sum_{i=1}^d i^{-\alpha\beta} \le C^{\alpha}\int_{\frac{1}{2}}^{d+\frac{1}{2}} x^{-\alpha\beta}\mathrm{d} x\\
            &=\begin{cases}
                \frac{C^\alpha}{1-\alpha\beta}\left(\left(d+\frac{1}{2}\right)^{1-\alpha\beta} - \left(\frac{1}{2}\right)^{1-\alpha\beta}\right),\quad&\alpha\beta\ne 1,\\
                C^\alpha \log(2d+1),\quad&\alpha\beta=1.
            \end{cases}
        \end{split}
    \end{equation}
\end{proof}

\setcounter{temp}{\value{theorem}}
\setcounter{theorem}{\value{pro:ridgeseparable}}
\begin{proposition}
    For the objective in \eqref{equ:ridgeseparable} that satisfies Assumptions \ref{asp: activation} and \ref{asp: data}, we have
    \begin{equation}
       \effdim_{\alpha} \leq
       \begin{cases}
      (L_0R)^{\alpha}, &\quad \alpha \geq 1, \quad \textit{dimensional free},\\
      (L_0R)^{\alpha}d^{1-\alpha},&\quad \alpha < 1\quad  \textit{improve by a } \Theta\left(d^{\alpha}\right) \textit{factor}.
       \end{cases}
    \end{equation}
\end{proposition}
\setcounter{theorem}{\value{temp}}
\begin{proof}
    First, we compute $\effdim_1$ as follows:
    \begin{equation}
        \begin{split}
            \effdim_1 &= \sum_{i=1}^d \sigma_i(\nabla^2 f)\\
            &= \left\|\sum_{i=1}^N\frac{1}{N} q''(\bbeta_i^\top\x)\bbeta_i\bbeta_i^\top\right\|_*\\
            &\leq \frac{L_0}{N}\left\|\sum_{i=1}^N\bbeta_i\bbeta_i^\top\right\|_*\le L_0R.
        \label{equ:ridgeseparable1}
        \end{split}
    \end{equation}
    For $\alpha\ge 1$, by the convexity of $g(x) = x^\alpha$, we have
    \begin{equation}
    \begin{split}
        \sum_{i=1}^d \sigma_i^\alpha(\nabla^2 f(\x)) \le \left(\sum_{i=1}^d \sigma_i(\nabla^2 f(\x))\right)^\alpha.
    \end{split}
    \label{equ:alphabig}
    \end{equation}
    For $\alpha< 1$, by H\"{o}lder's inequality, we have
    \begin{equation}
        \left(\sum_{i=1}^d \sigma_i^\alpha(\nabla^2 f(\x))\right)\le \left(\sum_{i=1}^d \sigma_i(\nabla^2 f(\x))\right)^{\alpha}\cdot d^{1-\alpha}.
        \label{equ:alphasmall}
    \end{equation}
    Taking supremum to both sides of \eqref{equ:alphabig} and \eqref{equ:alphasmall} on $\x$ yields the result.
\end{proof}

\setcounter{temp}{\value{theorem}}
\setcounter{theorem}{\value{pro:2nn}}
\begin{proposition}
    Define $f(\W, \w)= \w^\top \sigma(\W^\top\x) $, where $\sigma$ is the activation function. When $\|\x\|_1 \leq r_1$, $\|\w\| \leq r_2$ and $\sigma''(x)\leq \alpha$, we have 
        $\tr\left(\nabla^2 f(\W,\w)\right) \leq \alpha r_1 r_2$.
\end{proposition}
\setcounter{theorem}{\value{temp}}

\begin{proof}[Proof of Proposition \ref{prop:2nn}]
    By direct computation, we have
    \begin{equation}
    \begin{aligned}
        \frac{\partial f}{\partial \w} &= \sigma(\W^\top\x),\\
        \frac{\partial f}{\partial \W} &= \left(\sigma'(\W^\top\x)\odot \w\right) \otimes \x,\\
        \frac{\partial^2 f}{\partial \w^2}&= \mathbf 0,\\
        \frac{\partial^2 f}{\partial \W^2} &= \mathop{\mathrm{Diag}}(\sigma''(\W^\top \x)\odot \w)\otimes\x\otimes\x.
    \end{aligned}
    \end{equation}
    Therefore, 
    \begin{equation}
        \begin{aligned}
            \tr\left(\nabla^2 f(\W,\w)\right)) &= \|\x\|^2\cdot \tr\left(\mathop{\mathrm{Diag}}(\sigma''(\W^\top \x)\odot \w)\right)\\
            &\le r_1^2\cdot \langle\sigma''(\W^\top \x), \x \rangle\\
            &\le \alpha r_1r_2.
        \end{aligned}
    \end{equation}
\end{proof}

\setcounter{temp}{\value{theorem}}
\setcounter{theorem}{\value{lem:quadraticsubproblem}}
\begin{lemma}\label{lem:quadraticsubproblem}
For function $f$ that has $L$-continuous gradient and $M$-continuous Hessian matrices, given any $\delta>0$,
    Algorithm \ref{alg:quadraticsubproblem} outputs a $\delta$-approximated $f_\x(\y)$ such that $\left|\tilde f_{\x,\delta}(\y)- f_\x(\y)\right|\leq \delta$.
\end{lemma}
\setcounter{theorem}{\value{temp}}

\begin{proof}[Proof of Lemma \ref{lem:quadraticsubproblem}]
    We only need to prove that $|\tilde f_{\x, \delta}(\y)-f_\x(\y)|\le \delta$. We have the following inequality:
    \begin{equation}
        \begin{split}
            &\quad \left| \frac{Lr^2}{\delta}\left(f\left(\x+\frac{\delta}{Lr^2}(\y-\x)\right)-f(\x)\right) - \langle\nabla f(\x), \y-\x \rangle\right|\\
            &= \frac{Lr^2}{\delta}\left| f\left(\x+\frac{\delta}{Lr^2}(\y-\x)\right)-\left(f(\x) - \langle\nabla f(\x), \frac{\delta}{Lr^2}(\y-\x) \rangle\right)\right|\\
            &\stackrel{a}{\le} \frac{Lr^2}{\delta}\cdot \frac{L}{2}\left\| \frac{\delta}{Lr^2}(\y-\x)\right\|^2\\
            &= \frac{Lr^2}{\delta}\cdot \frac{L}{2}\left(\frac{\delta}{Lr}\right)^2= \frac{\delta}{2},
            \label{equ:quadraticsubproblempf1}
        \end{split}
    \end{equation}
    where $\stackrel{a}{\le}$ uses the $L$-Lipschitz continuity of $\nabla f(\x)$. We also have
    \begin{align}
         &\!\!\Bigg|\frac{2H^2r^6}{\delta^2}\left(f\left(\x+\frac{\delta}{2Hr^3}(\y-\x)\right)+f\left(\x-\frac{\delta^2}{2Hr^3}(\y-\x)\right)-2f(\x)\right) - \langle \nabla^2f(\x)(\y-\x), \y-\x\rangle\Bigg|\notag\\
         &\!\!\!\!\!\!\!\!\stackrel{a}{=} \!\left| \frac{1}{2}(\langle \nabla^2f(\x_1)(\y-\x), \y-\x\rangle + \langle \nabla^2f(\x_2)(\y-\x), \y-\x\rangle) - \langle \nabla^2f(\x)(\y-\x), \y-\x\rangle\right|,
\label{equ:quadraticsubproblempf2} 
    \end{align}
    where $\stackrel{a}{=}$ uses the second-order Taylor expansion of $f$ at $\x$. We have $\|\x_1-\x\|\le \frac{\delta}{2Hr^2}$ and $\|\x_2-\x\|\le \frac{\delta}{2Hr^2}$. By the $H$-Lipschitz continuity of $\nabla^2 f(\x)$, we have
    \begin{align}
             &\left| \frac{1}{2}(\langle \nabla^2f(\x_1)(\y-\x), \y-\x\rangle + \langle \nabla^2f(\x_2)(\y-\x), \y-\x\rangle) - \langle \nabla^2f(\x)(\y-\x), \y-\x\rangle\right|\notag\\
            &\le \frac{\delta}{2Hr^2}\cdot Hr^2 = \frac{\delta}{2}.
        \label{equ:quadraticsubproblempf3}  
    \end{align}
    By \eqref{equ:quadraticsubproblempf1}, \eqref{equ:quadraticsubproblempf2}, and \eqref{equ:quadraticsubproblempf3} we have $|\tilde f_{\x, \delta}(\y)-f_\x(\y)|\le \delta$, hence the lemma is proved.
\end{proof}

\begin{proof}[Proof of Corollary \ref{cor:stronglyconvex}]
    By \eqref{equ:ridgeseparable1} and Theorem \ref{thm:RGconvex}, we know that Algorithm \ref{alg:RG} needs $\tO\left(\frac{L_0R}{\mu}\right)$ to find an $\epsilon$-approximated solution with high probability.
\end{proof}

\begin{proof}[Proof of Corollary \ref{cor:weaklyconvex}]
    By \eqref{equ:ridgeseparable1} and Theorem \ref{thm:RGweaklyconvex}, we know that Algorithm \ref{alg:RG} needs $\tO\left(\frac{L_0R}{\epsilon}\right)$ to find an $\epsilon$-approximated solution with high probability.
\end{proof}

\end{document}